\documentclass{amsart}
\usepackage{a4,amssymb,amsmath,mathrsfs,graphicx}

\numberwithin{equation}{section}

\newtheorem{thm}{Theorem}[section]
\newtheorem{lem}[thm]{Lemma}
\newtheorem{prop}[thm]{Proposition}
\newtheorem{cor}[thm]{Corollary}

\theoremstyle{definition}

\theoremstyle{remark}
 \newtheorem{rmk}[thm]{Remark}

\newcommand{\R}{\mathbb{R}}

\newcommand{\C}{{\mathbb C}}

\newcommand{\hn}{{\mathbb{H}^n}}
\newcommand{\Cn}{{\mathbb C}^n}
\newcommand{\Rn}{\mathbb{R}^{n}}

\newcommand{\diam}{\operatorname{diam}}
\newcommand{\im}{\operatorname{Im}}
\newcommand{\leb}{\mathcal{L}^{2n+1}}
\newcommand{\lone}{\mathcal{L}^{1}}
\newcommand{\tco}{t\textup{-co}\:}
\newcommand{\st}{\operatorname{St}}
\newcommand{\I}{\mathcal{I}}
\newcommand{\rot}{\mathcal{R}}
\newcommand{\interior}{\operatorname{int}}
\newcommand{\spa}{\operatorname{span}}
\newcommand{\p}{\partial}
\newcommand{\scal}[2]{\langle{#1},{#2}\rangle}
\newcommand{\ball}{\mathscr B}
\newcommand{\lip}{\operatorname{Lip}}

\newcommand{\si}{\sigma}
\newcommand{\vp}{\varphi}
\newcommand{\eps}{\varepsilon}

\begin{document}

\title[Isodiametric sets in the Heisenberg group]{Isodiametric sets in the Heisenberg group}

\author[G.P. Leonardi]{G.P. Leonardi}
\address[G.P. Leonardi]{Universit\`a di Modena e Reggio Emilia,
  Dipartimento di Matematica Pura ed Applicata, via Campi 213/b, 41100
  Modena, Italy}
\email{gianpaolo.leonardi@unimore.it}

\author[S. Rigot]{S. Rigot}
\address[S. Rigot]{Universit\'e de Nice Sophia-Antipolis, Laboratoire J.-A. Dieudonn\'e, CNRS-UMR 6621, Parc Valrose, 06108 Nice cedex 02, France}
\email{rigot@unice.fr}

\author[D. Vittone]{D. Vittone}
\address[D. Vittone]{Universit\`a di Padova, Dipartimento di Matematica Pura ed Applicata, via Trieste 63, 35121 Padova, Italy}
\email{vittone@math.unipd.it}
\thanks{The first and third authors have been supported by E.C. 
project ``GALA'', MIUR, GNAMPA project ``Metodi geometrici per analisi in
  spazi non Euclidei: spazi metrici doubling, gruppi di Carnot e spazi
  di Wiener'' (2009) and, respectively, by the University of Modena
  and Reggio Emilia and the University of Padova, Italy.} 
\thanks{The second author wishes to thank the Department of Pure and Applied
  Mathematics, University of Modena and Reggio Emilia, where 
  part of the work was done, and also the GNAMPA project for financial support. The first and third authors are pleased to thank the
  Laboratoire J.-A. Dieudonn\'e, Universit\'e de Nice
  Sophia-Antipolis, for the hospitality during the completion of a
  first draft of the paper.}



\begin{abstract}
In the sub-Riemannian Heisenberg group equipped with its
Car-\- not-Carath\'eodory metric and with a Haar measure, we
consider \textit{isodiametric sets}, i.e. sets maximizing the measure
among all sets with a given diameter. In particular, given an
isodiametric set, and up to negligible sets, we prove that its
boundary is given by the graphs of two locally Lipschitz
functions. Moreover, in the restricted class of \textit{rotationally invariant
sets}, we give a quite complete characterization of any compact
(rotationally invariant) isodiametric set. More specifically, its
Steiner symmetrization with respect to the $\C^n$-plane is
shown to coincide with the Euclidean convex hull of a CC-ball. At the
same time, we also prove quite unexpected non-uniqueness results.
\end{abstract}

\maketitle

\section{Introduction}
\label{sec:intro}

The classical isodiametric inequality in the Euclidean space says that
balls maximize the volume among all sets with a given diameter. This
was originally proved by Bieberbach~\cite{bieberbach} in $\R^2$ and
by Urysohn~\cite{urysohn} in $\R^n$, see
also~\cite{burago-zalgaller}. In this paper we are interested in the
case of the Heisenberg group $\hn$ equipped with its
Carnot-Carath\'eodory distance $d$ and with the Haar measure $\leb$
(see Section~\ref{sec:heisenberg} for the definitions). Our aim is to study \textit{isodiametric sets},
  i.e. sets maximizing the measure among sets with a given diameter.
\smallskip

Recalling that the homogeneous dimension of $\hn$ is $2n+2$, we define the isodiametric constant $C_I$ by 
\begin{equation*}
 C_I = \sup \leb (F)/(\diam F)^{2n+2}
\end{equation*}
where the supremum is taken among all sets $F\subset \hn$ with
positive and finite diameter. Sets realizing the supremum do exist,
see~\cite{rigot} or Theorem~\ref{thm:existence} below. Since the
closure of any such set is a compact set that still realizes the
supremum, we consider the class $\I$ of compact isodiametric sets,
\begin{equation*}
 \I = \{ E\subset \hn~;~E\text{ compact},~\diam E>0,~\leb (E) = C_I~(\diam E)^{2n+2} \}~.
\end{equation*}
In other words, due to the presence of dilations in $\hn$, $\I$ denotes the class of compact sets that maximize the $\leb$-measure among all sets with the same diameter.
\smallskip

In contrast to the Euclidean case, balls in $(\hn,d)$ are not isodiametric (\cite{rigot}) and we shall give in this paper some further and refined evidence that the situation is indeed quite different from the Euclidean one.
\smallskip

Before describing our main results let us recall some classical motivations and consequences coming from the study of isodiametric type problems. First the isodiametric constant $C_I$ coincides with the ratio between the measure $\leb$ and the $(2n+2)$-dimensional Hausdorff measure $\mathcal{H}^{2n+2}$ in $(\hn,d)$, namely,
\begin{equation*} 
 \leb = C_I~\mathcal{H}^{2n+2}~,
\end{equation*}
where $\mathcal{H}^{2n+2}(A) = \lim_{\delta\downarrow 0} \inf \left\{ \sum_i (\diam A_i)^{2n+2}~;~A\subset \cup_i A_i~,~\diam A_i\leq \delta\right\}$. This can actually be generalized to any Carnot group equipped with a homogeneous distance (\cite{rigot}), and for abelian Carnot groups one recovers the well-known Euclidean situation. 
\smallskip

As a consequence, the knowledge of the numerical value of the isodiametric constant $C_I$, or equivalently the explicit description of isodiametric sets, gives non trivial information about the geometry of the metric space $(\hn,d)$ and about the measure $\mathcal{H}^{2n+2}$ which may be considered as a natural measure from the metric point of view.
\smallskip

Let us also mention that there are some links with the Besicovitch 1/2-problem which is in turn related to the study of the connections between densities and rectifiability. See~\cite{preisstiser} for an introduction and known results about the Besicovitch 1/2-problem and \cite{rigot} for the connection between the isodiametric problem in Carnot groups and the Besicovitch 1/2-problem.
\smallskip

Our main results in the present paper are a regularity property for sets in $\I$ and a rather complete solution to a restricted  isodiametric problem within the class of so-called rotationally invariant sets.
\smallskip

Let us first describe our regularity result. We shall prove that given $E\in\I$ then $\overline{\interior E}$ is still a compact isodiametric set with the same diameter as $E$ and with locally Lipschitz boundary. More precisely, identifying $\hn$ with $\Cn \times \R$ (see Section \ref{sec:heisenberg}), we prove that
\begin{equation*}
\overline{\interior E} = \{[z,t]\in \hn~;~z\in \overline U,~f^-(z)\leq t \leq f^+(z)\}
 \end{equation*}
for some open set $U$ in $\Cn$ and some continuous maps $f^-$, $f^+: \overline U \rightarrow \R$ that are locally Lipschitz continuous on $U$. See Theorem~\ref{thm:regisodiam} for a complete statement.
\smallskip

This regularity property will actually follow from a slightly more general result. We will prove that a set $E\in\I$ must satisfy the following necessary condition, 
\begin{equation} \tag{NC} \label{nc}
 \text{for all } p\in \p E, \text{ there exists } q \in \p E \text{ such that } d(p,q) = \diam E~,
\end{equation}
see Proposition~\ref{prop:neccdtrefined}, and is $t$-convex, see Subsection~\ref{subsec:geometric-transf} for the definition of $t$-convexity and Proposition~\ref{prop:tconv-isodiam}. Independently from the isodiametric problem, the property \eqref{nc} together with $t$-convexity turn out to imply the regularity properties sketched above. See Theorem~\ref{thm:regboundary}. 
\smallskip

As already mentioned, one knows that balls in $(\hn,d)$ are not isodiametric and isodiametric sets in $(\hn,d)$ are actually not explicitly known so far. This question turns out to be a challenging and rather delicate one. However, restricting ourselves to the family $\rot$ of so-called rotationally invariant sets, we are able to give a rather complete picture of the situation for compact isodiametric sets within this class. As we shall explain below, this picture will give some further information about the class $\I$. This may also hopefully give some insight towards a complete solution of the general isodiametric problem in $(\hn,d)$.
\smallskip

We shall denote by $\I_\rot$ the class of compact sets in $\rot$ that
are isodiametric within the class $\rot$. See
Section~\ref{sec:isodiampb} for the definition of the class $\rot$ and
of this restricted isodiametric problem. First it is not hard
to check that sets in $\I_\rot$ satisfy \eqref{nc} and are $t$-convex, see
Proposition~\ref{prop:neccdtrefined} and Proposition~\ref{prop:tconv-isodiam}, and hence satisfy the regularity
properties of Theorem~\ref{thm:regboundary}. Next our main
specific result concerning sets $E\in \I_\rot$ is the characterization
of their Steiner symmetrization $\st E$ with respect to the
$\Cn$-plane (see Subsection~\ref{subsec:geometric-transf} for the
definition of $\st E$). We prove that if $E\in \I_\rot$ then $\st E$
belongs to $\I_\rot$ and has the same diameter as $E$. Moreover we
prove that $\st E$ is actually uniquely determined once the diameter
of $E$ is fixed, i.e., $\st E= A_{\diam E}$ for some peculiar set $A_{\diam E}$, see Theorem~\ref{thm:rotisodiamsets}.
\smallskip

Given $\lambda>0$, the set $A_\lambda$ can be guessed via the
following argument. One starts with the ball in $(\hn,d)$ centered at
the origin and with diameter $\lambda$. As already said it does not
satisfy \eqref{nc}. Thus one can enlarge it around points where \eqref{nc} fails and get a set with still the same diameter but
with greater measure. One can actually try to enlarge it as much as
possible without increasing the diameter, remaining in the class $\rot$, and preserving the property that it coincides with its Steiner
symmetrization with respect to the $\Cn$-plane. In such a way,
one ends up with a maximal set $A_\lambda$ that satisfies \eqref{nc}. It turns out that this set is the closed convex hull (in the Euclidean sense when identifying $\hn$ with $\R^{2n+1}$) of the ball in $(\hn,d)$ centered at the origin and with diameter $\lambda$. See Figure~\ref{ccball}.
\begin{figure}[h!tbp]
  \centering
  \includegraphics[width=5cm]{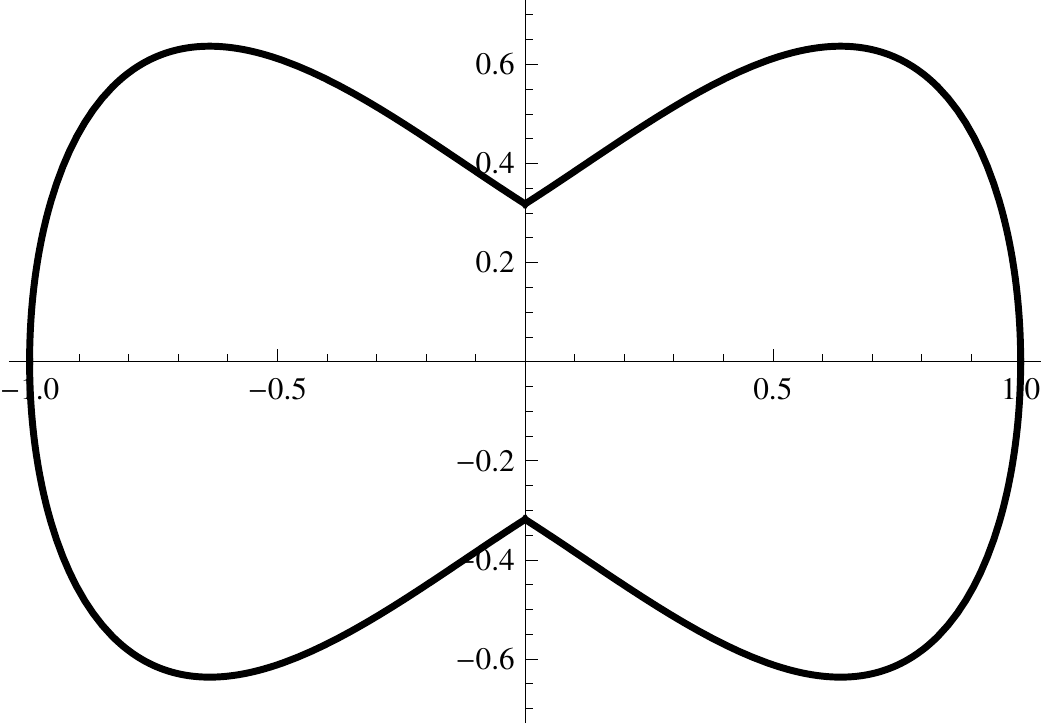} \qquad
  \includegraphics[width=5cm]{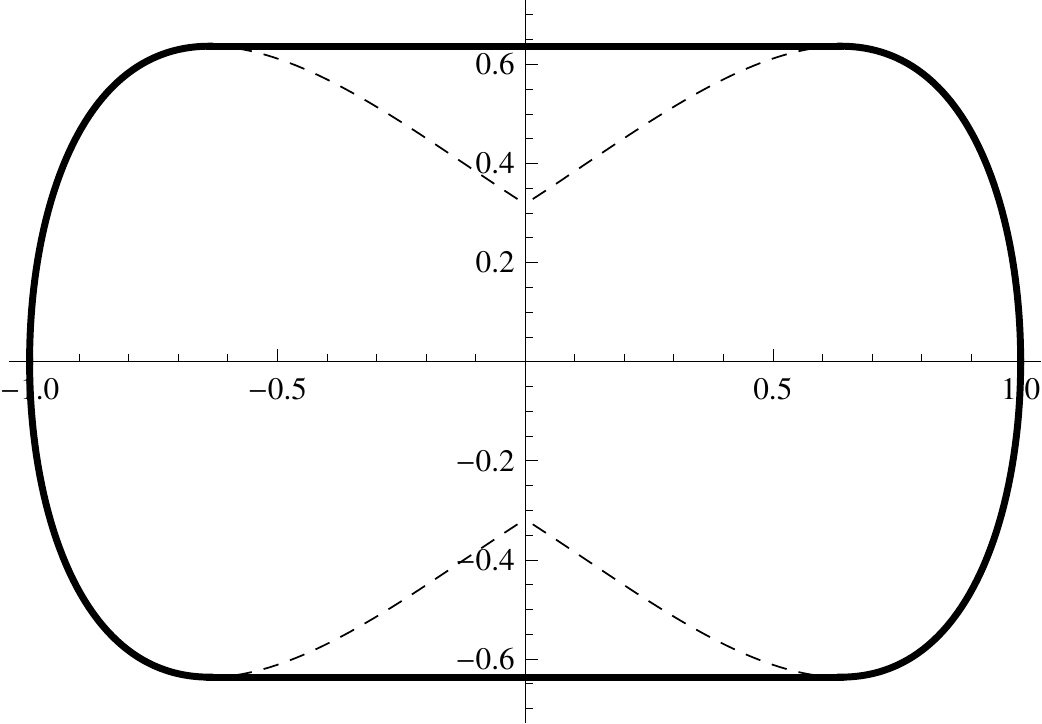}
  \caption{The intersection of the boundary $\partial
    B(0,1)$ of the unit ball (left) and of the boundary $\partial
    A_2$ of its Euclidean convex hull (right) with a plane containing the vertical $t$-axis.}
  \label{ccball}
\end{figure}

\smallskip

We also construct small suitable perturbations of the set $A_\lambda$ that preserve its Lebesgue measure and its diameter, see Proposition~\ref{prop:Aphi}. Considering rotationally invariant perturbations, this gives the non uniqueness of sets in $\I_\rot$. This non uniqueness has to be understood in an ``essential'' sense, i.e., also up to left translations and dilations. See Corollary~\ref{cor:nonuniqueness}.
\smallskip

Finally, considering non rotationally invariant pertubations, one gets the existence of sets in $\I$ that are not rotationally invariant, even modulo left translations, see Corollary~\ref{cor:nonrotinvariant}. This gives one more significant difference with the Euclidean situation. See Figure~\ref{A2pert}.
\smallskip
\begin{figure}[h!tbp]
  \centering
  \includegraphics[width=5cm]{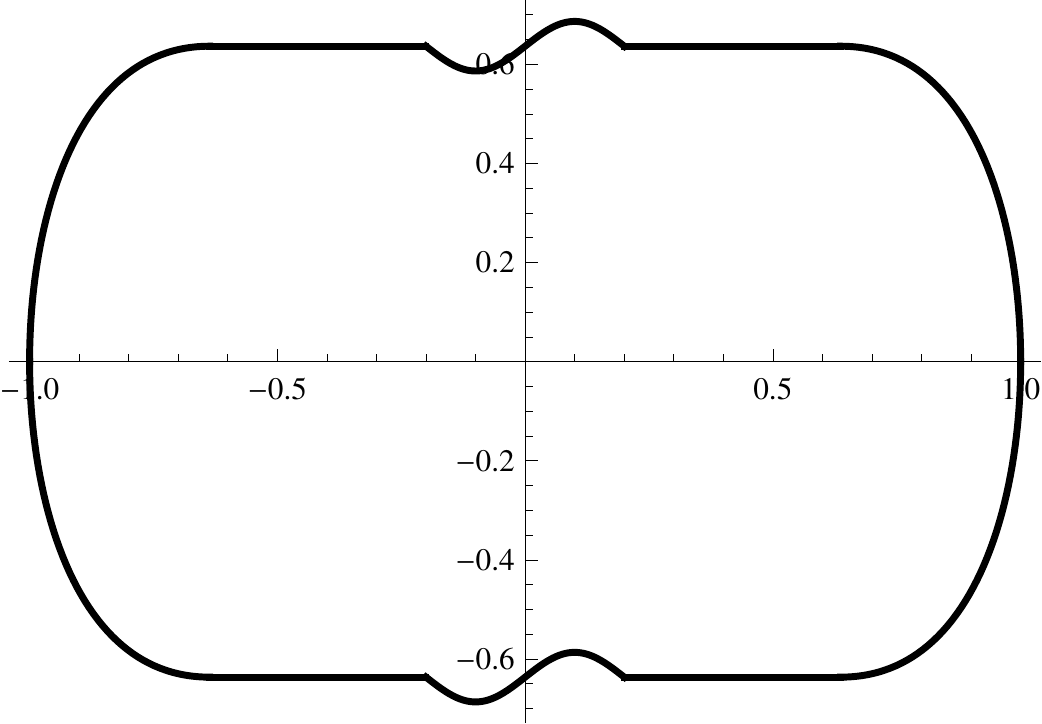}
  \caption{The intersection of the boundary of a perturbation of $A_2$ with
    a plane containing the vertical $t$-axis.} 
  \label{A2pert}
\end{figure}
\smallskip

The paper is organized as follows. In Section~\ref{sec:heisenberg} we introduce the Heisenberg group $\hn$ and recall basic facts about the Carnot-Carath\'eodory distance $d$ and balls in $(\hn,d)$. We also introduce $t$-convexification and the Steiner symmetrization with respect to the $\Cn$-plane. Section~\ref{sec:isodiampb} is devoted to existence and regularity results, while in Section~\ref{sec:rotisodiamsets} we prove our more specific results about the isodiametric problem restricted to the class $\rot$.

\section{Notations and preliminary results}
\label{sec:heisenberg}

\subsection{The Heisenberg group}
\label{subsec:heis}
The Heisenberg group $\hn$ is a connected and simply connected Lie
group with stratified Lie algebra (see e.g. \cite{stein,CDPT}). We identify it with $\Cn \times \R$
and denote points in $\hn$ by $[z,t]$, where $z=(z_1,\dots,z_n) \in
\Cn$ and $t\in\R$. The group law is 
\begin{equation*}
[z,t]\cdot [z',t'] = [z + z', t+t'+ 2 \im z \overline z' ]
\end{equation*}
where $\im z \overline z'=\sum_{j=1}^n \im z_j \overline z'_j$. The unit element is $0=[0,0]$. 

There is a natural family of dilations $\delta_\lambda$ on $\hn$ defined by $\delta_\lambda([z,t])= [ \lambda z ,\lambda^2 t]$ for $\lambda\geq 0$.

We define the canonical projection $\pi:\hn\to \Cn$ as 
\begin{equation} \label{e:projection}
\pi([z,t]) = z
\end{equation}
for any $[z,t]\in \hn$. Given $p\in\hn$, we
shall frequently denote by $[z_p,t_p]$ its coordinates in $\Cn\times\R$. 

We also represent
$\hn$ as $\R^{2n+1}\approx \Rn \times \Rn \times \R$ through the
identification $[z,t]\approx (x,y,t)$, where $x=(x_1,\dots,x_n)$, 
$y=(y_1,\dots,y_n) \in \R^n$, $t\in \R$ and $z =
(z_1,\dots,z_n)\in\Cn$ with $z_j = x_j + i y_j$. 

The $(2n+1)$-dimensional Lebesgue measure $\leb$ on $\hn \approx \R^{2n+1}$ is a Haar measure of the group. It is $(2n+2)$-homogeneous with respect to dilations,
\begin{equation*}
 \leb(\delta_\lambda(F)) = \lambda^{2n+2} \leb(F)
\end{equation*}
for all measurable $F\subset \hn$ and $\lambda\geq 0$. 

The horizontal subbundle of the tangent bundle is defined by
\begin{equation*} 
\mathcal{H}^n = \spa \left\{X_j,  Y_j~;~j=1,\dots,n\right\} 
\end{equation*}
where the left invariant vectors fields $X_j$ and $Y_j$ are given by
\begin{equation*}
X_j = \partial_{x_j} + 2 y_j \partial_t \,, \quad Y_j = \partial_{y_j} - 2 x_j \partial_t.
\end{equation*}
Setting $T=\partial_t$ the only non trivial bracket relations are $[X_j,Y_j] =-4T$, hence the Lie algebra of $\hn$ admits the stratification $\mathcal{H}^n \oplus \spa \{T\}$.

\subsection{Carnot-Carath\'eodory distance}
\label{subsec:dc}
We fix a left invariant Riemannian metric $g$ on $\hn$ that makes $(X_1,\dots,X_n,Y_1,\dots,Y_n,T)$ an orthonormal basis. The Carnot-Carath\'eodory distance between any two points $p$ and $q \in \hn$ is then defined by 
\begin{equation*}
 d(p,q) = \inf \{ length_{g} (\gamma); \; \gamma \text{ horizontal curve joining } p \text{ to } q\}
\end{equation*}
where a curve is said to be horizontal if it is absolutely continuous
and such that at a.e.~every point its tangent vector belongs to the
horizontal subbundle $\mathcal H^n$ of the tangent bundle. Recall that by Chow-Rashevsky
theorem any two points can be joined by a horizontal curve of finite
length. Therefore the function $d$ turns out to be a distance. It induces the
original topology of the group, it is left-invariant, i.e., 
\begin{equation*}
 d(p\cdot q,p\cdot q')=d(q,q')
\end{equation*}
for all $p$, $q$, $q' \in \hn$, and one-homogeneous with respect to
dilations, i.e.,
\begin{equation*}
 d(\delta_\lambda(p),\delta_\lambda(q))=\lambda~d(p,q)
\end{equation*}
for all $p$, $q \in \hn$ and $\lambda\geq 0$. 
\smallskip

Equipped with this distance, $\hn$ is a separable and complete metric
space in which closed bounded sets are compact. We will denote by
$B(p,r)$, respectively $\overline B(p,r)$, the open, respectively
closed, ball with center $p\in\hn$ and radius $r>0$. Note that the
diameter of any ball in $(\hn,d)$ is given by twice its radius. 
\smallskip

\begin{lem}\label{lem:distopen}
 Let $p\in\hn$. The distance function $d_p:\hn\setminus\{p\} \rightarrow \R$ from $p$ defined by $d_p(q) = d(p,q)$ is an open map.
\end{lem}

\begin{proof}
We prove that $d_p(B)$ is open for any open ball $B\subset
\hn\setminus \{p\}$. Since balls in $(\hn,d)$ are connected (this is
more generally true in any length space) it follows that $d_p(B)$ is a
bounded interval. Setting $m = \inf(d_p(q)~;~q\in B)$ and $M =
\sup(d_p(q)~;~q\in B)$ it is thus enough to prove that $m \not \in
d_p(B)$ and $M\not \in d_p(B)$. 

If $m=0$ then $m\notin d_p(B)$ because $p\notin B$. If $m>0$
we assume by contradiction that $m \in d_p(B)$. Then we can find
$q\in B$ such that $d_p(q) = m$. The map $\lambda \in [0,+\infty)
\mapsto p \cdot \delta_\lambda(p^{-1} \cdot q)$ being continuous, one
has $p \cdot \delta_\lambda(p^{-1} \cdot q) \in B$ for all $\lambda$
close enough to 1. It follows that 
\begin{equation*}
 d_p(q) \leq d_p(p \cdot \delta_\lambda(p^{-1} \cdot q)) = \lambda \, d_p(q) < d_p(q)
\end{equation*}
provided $\lambda <1$ is close enough to 1, which gives a contradiction. The fact that $M\not \in d_p(B)$ can be proved in a similar way and this concludes the proof.
\end{proof}

\begin{rmk} \label{rmk:distopen}
As an immediate consequence of Lemma~\ref{lem:distopen} we get that 
for any set 
$F\subset \hn$ and any $p\in F$,  
\[
d(p,q)<\diam F \;\; \text{for all $q\in \interior F$. }
\]
Moreover, if $F$ is bounded then $\diam F
= \diam \partial F$.
\end{rmk}

Although the Carnot-Carath\'eodory distance between any two points is
in general hardly explicitly computable, we recall for further
reference the following well-known peculiar cases. One has 
\begin{equation} \label{e:d1}
 d([z,t],[z',t]) = \|z'-z\| 
\end{equation}
for all $z$, $z'\in\Cn$ such that $\im z \overline z' =0$ and all $t\in\R$. Here $\|z\| = (\sum_{j=1}^n |z_j|^2)^{1/2}$ for $z=(z_1,\dots,z_n)\in\Cn$. We have also
\begin{equation} \label{e:d2}
 d([z,t],[z,t']) = (\pi |t'-t|)^{1/2} 
\end{equation}
for all $z\in\Cn$ and $t$, $t'\in\R$.

\subsection{Description of balls and consequences} \label{subsec:balls}

Explicit descriptions of balls in $(\hn,d)$ are well-known,
see~\cite{ambrig},~\cite{gaveau},~\cite{juillet}. One has 
\begin{multline*}
\overline B(0,1)  = \\ \Big\{\Big[\dfrac{\sin\vp}{\vp}\, \chi~,~\dfrac{2\vp - \sin(2\vp)}{2\vp^2}\,\|\chi\|^2\Big]\in\hn~;~\chi\in\Cn~,~\|\chi\|\leq 1~,~\vp\in[-\pi,\pi]\Big\}~.
\end{multline*}
\smallskip

We set 
\begin{equation*}
g(\vp) = 
\begin{cases}
          \dfrac{2\vp - \sin(2\vp)}{2\vp^2} & \text{for } 0<\vp \leq \pi\\
          0 & \text{if } \vp = 0~.
         \end{cases}
\end{equation*}
The function $g$ admits a maximum at $\vp = \pi/2$ with $g(\pi/2) = 2/\pi$. It is increasing from $[0,\pi/2]$ onto $[0,2/\pi]$ and decreasing from $[\pi/2,\pi]$ onto $[1/\pi,2/\pi]$. 
\smallskip

We set
\begin{equation*}
\rho(\vp) = \begin{cases}
           \dfrac{\sin\vp}{\vp} \quad \text{for } 0<\vp \leq \pi \\
          1 \phantom{\dfrac{\sin u}{\vp}} \quad\text{if } \vp = 0~.
         \end{cases}
\end{equation*}
The function $\rho$ is decreasing from $[0,\pi]$ onto $[0,1]$. Let
$\rho^{-1}:[0,1] \rightarrow [0,\pi]$ denote its inverse and set
\begin{equation}
  \label{eq:defacca}
h = g \circ \rho^{-1}.
\end{equation}
\smallskip

We have the following other  description of the closed unit ball 
\begin{equation}\label{profile}
\overline B(0,1) = \{[z,t]\in\hn~;~ \|z\|\leq 1,~|t| \leq h(\|z\|)\}~.
\end{equation}
Using dilations we have 
\begin{equation*}
\overline B(0,\lambda) = \{[z,t]\in\hn~;~ \|z\|\leq \lambda,~|t| \leq h_\lambda(\|z\|)\}
\end{equation*}
for all $\lambda>0$, where 
\begin{equation}
  \label{eq:accaerre}
h_\lambda(\|z\|)=\lambda^2 h(\|z\|/\lambda).  
\end{equation}

\begin{figure}[h!tbp]
  \centering
  \includegraphics[width=5cm]{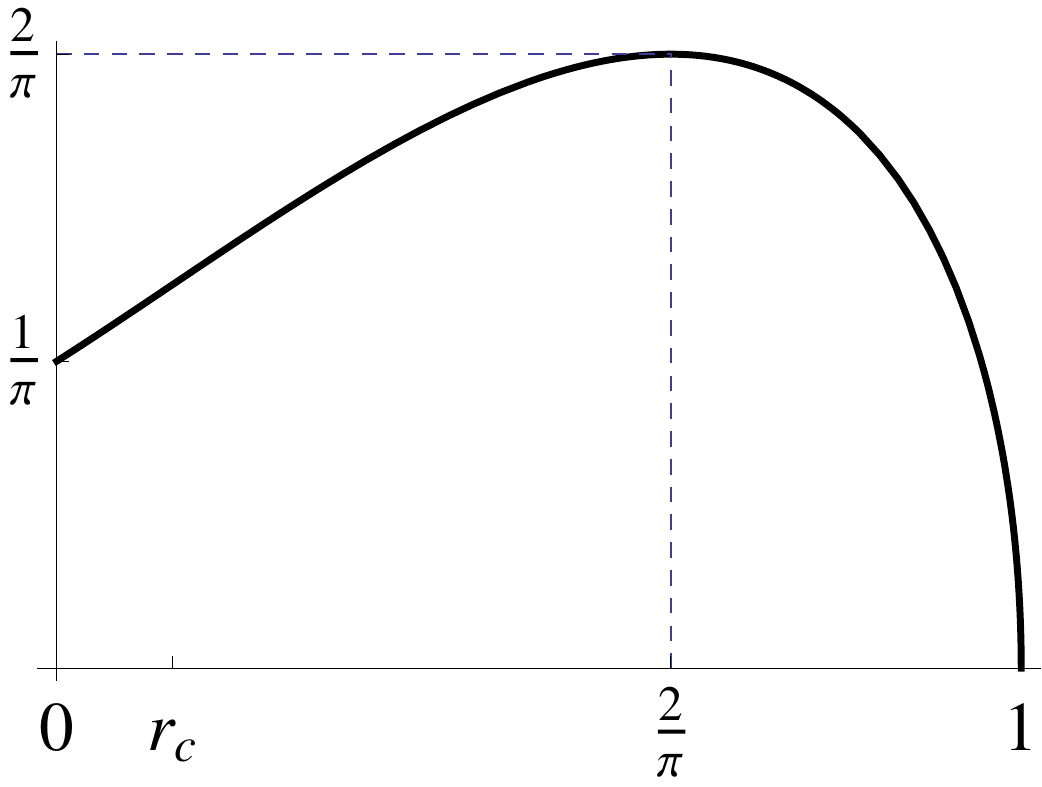}
  \caption{The profile function $h$.}
  \label{grafacca}
\end{figure}
\smallskip

We list some properties of the function $h$ that will
be needed in the sequel. See Figure~\ref{grafacca} for a picture. 
\begin{itemize}
\item[(i)] The map $h$ is increasing from $[0,2/\pi]$ onto
$[1/\pi,2/\pi]$ and decreasing from $[2/\pi,1]$ onto $[0,2/\pi]$.
\smallskip

\item[(ii)] There exists $r_c \in (0,2/\pi)$ such that $h''>0$ on $(0,r_c)$,
  $h''(r_c)=0$, and $h''<0$ on $(r_c,1)$.
\end{itemize}
Indeed we have 
\begin{equation*}
 h'(r) = -2 \, \dfrac{\cos \vp}{\vp}
\end{equation*}
and
\begin{equation*}
h''(r) = 2 \, \dfrac{\vp \sin\vp + \cos\vp}{\vp^2 \rho'(\vp)}
\end{equation*}
for all $r\in [0,1)$ and where $\vp = \rho^{-1}(r) \in
(0,\pi]$. Then statement (i) follows from the expression of $h'$ together with the properties of $\rho^{-1}$. Statement (ii) follows noting
that there exists a unique $\vp_c\in(\pi/2,\pi)$ such that $\vp_c
\sin\vp_c + \cos\vp_c=0$ and that one has $\vp
\sin\vp + \cos\vp>0$ if and only if $0<\vp < \vp_c$. 
\smallskip

We call vertical segment any set of the form $\{[z,t]\in\hn~;~ t^-\leq
t\leq t^+\}$ for some $t^-\leq t^+$. Given $p$, $q\in\hn$ with $z_p =
z_q$, we denote by $L_{p,q}$ the vertical segment joining $p$ to $q$, 
\begin{equation*}
L_{p,q} = \{[z_p,t]\in\hn~;~\min(t_p,t_q)\leq t \leq \max(t_p,t_q)\}~. 
\end{equation*}
In the next proposition we state an elementary geometric property of balls in $\hn$ for further reference. When not specified, by ball we mean a ball that can be indifferently taken as open or closed.

\begin{prop} \label{prop:ballsproperties}
The following statements hold.
\begin{itemize}
\item[(i)] Let $B$ denote a ball in $\hn$. For any $p$, $q\in B$ such that $z_p = z_q$, we have $L_{p,q} \subset B$. 

\item[(ii)] For any $p \in \hn$ and any $p_1$, $p_2\in \hn$ such that $z_{p_1} = z_{p_2}$, we have 
\begin{equation*}
 d(p,q)\leq \max(d(p,p_1),d(p,p_2))
\end{equation*}
for all $q \in L_{p_1,p_2}$.
\end{itemize}
\end{prop}

\begin{proof}
 Property (i) holds for the (closed or open) unit ball by
 \eqref{profile}. Then this property follows for any ball using
 dilations and translations and noting that these maps are bijective
 maps that send vertical segments onto
 vertical segments. 

To prove property (ii) set $r=\max(d(p,p_1),d(p,p_2))$. We have
$p_1$, $p_2 \in \overline B(p,r)$. Hence $L_{p_1,p_2} \subset
\overline B(p,r)$ by (i) and thus (ii) follows.
\end{proof}
\smallskip

In the next lemma we deal with an outer vertical cone property for
balls centered at the origin. Its proof (that we provide for the
reader's convenience) follows from the local Lipschitz continuity of
the profile function $h$ on $[0,1)$. 

\begin{lem}  \label{coneball}
Let $d>0$ and $\delta>0$ be fixed. There exists $\alpha(d,\delta)>0$ such that the following holds. If $p\in \p B(0,d)$ is such that $t_p\geq \delta$, respectively $t_p\leq-\delta$, and $[w,s] \in \hn$ is such that 
\begin{equation*}
 s  > t_p +\alpha(d,\delta) \:\|w-z_p\|~,~ \text{respectively}~ s < t_p - \alpha(d,\delta) \:\|w-z_p\|~,
\end{equation*}
then $[w,s]\notin \overline B(0,d)$.
\end{lem}

\begin{proof}
Let $\alpha = \alpha(d,\delta)>0$ to be chosen later. Let us consider the case where $p\in \p B(0,d)$ and $[w,s] \in \hn$ are such that $t_p\geq \delta$ and $s  > t_p +\alpha\:\|w-z_p\|$, the other case being analogous. If $\|w\|>d$ then we obviously have $[w,s]\notin \overline B(0,d)$. If $\|w\|\leq d$ we want to prove that $s> h_d(\|w\|)$. We have
\[
s > t_p + \alpha\:\|w-z_p\| = h_d(\|z_p\|)+ \alpha\:\|w-z_p\|
\]
and thus it will be sufficient to show that 
\[
h_d(\|z_p\|)+ \alpha\:\|w-z_p\|\geq h_d(\|w\|)~.
\]

Since $\lim_{r\to d^-}h_d(r)=0$, one can find $\overline r=\overline r(d,\delta) \in (0,d)$ such that $h_d(r)< \delta$ for all $r\in (\overline r,d]$. In particular $\|z_p\|\leq \overline r$ because $h_d(\|z_p\|)=t_p\geq\delta$. We choose $\alpha >0$ to be the Lipschitz constant of $h_d$ on $[0,\overline r]$. If $\|w\| \leq \overline r$ it follows
\[
 h_d(\|w\|) \leq h_d(\|z_p\|)+ \alpha\:\big|\|w\|-\|z_p\|\big| \leq h_d(\|z_p\|)+ \alpha\:\|w-z_p\|
\]
as wanted. Whereas 
\[
h_d(\|w\|) <\delta \leq h_d(\|z_p\|)\leq h_d(\|z_p\|)+ \alpha\:\|w-z_p\|
\]
if $\|w\| \in (\overline r,d]$ which concludes the proof.
\end{proof}

\begin{rmk}\label{rem:alpha}
We note that, for any fixed $d>0$, the function $\alpha(d,\delta)$ can
be taken continuous w.r.t. the variable $\delta$. This follows from
the definition of
$\alpha(d,\delta)=\|h_d'\|_{L^\infty([0,\overline r(d,\delta)])}$ together
with the fact that $\delta \mapsto \overline r(d,\delta)$ can be chosen to be
continuous and the fact that the map $r \in (0,d) \mapsto
\|h_d'\|_{L^\infty([0,r])}$ is continuous. 
\end{rmk}
\smallskip

\subsection{Two geometric transformations in $\hn$.} \label{subsec:geometric-transf}
In this subsection we introduce two geometric transformations, namely
the convexification along the vertical $t$-axis and the Steiner
symmetrization with respect to the $\Cn$-plane. They will play a
crucial role in the sequel.
\smallskip

Given $F\subset\hn$ we define its $t$-convex hull $\tco F$ by
\begin{equation*}
\tco F =\{p\in \hn~;~p\in L_{p_1,p_2} \text{ for some } p_1,\, p_2 \in F \text{ with } z_{p_1} = z_{p_2}\}~.
\end{equation*}
We say that $F$ is $t$-convex if $F = \tco F$.

\begin{lem} \label{lem:tco}
 Let $F\subset \hn$. We have $F \subset \tco F$ and $\diam (\tco F) = \diam F$.
\end{lem}

\begin{proof}
We obviously have $F \subset \tco F$ and in particular $\diam F \leq
\diam (\tco F)$. Conversely let $p$, $p'\in \tco F$. One can find
$p_1$, $p_2 \in F$ with $z_{p_1} = z_{p_2}$ such that $p\in
L_{p_1,p_2}$. Then it follows from
Proposition~\ref{prop:ballsproperties}(ii) that 
\begin{equation*}
 d(p',p)\leq \max(d(p',p_1),d(p',p_2)).
\end{equation*}
Similarly one can find $p'_1$, $p'_2 \in F$ with $z_{p'_1} = z_{p'_2}$ such that $p'\in L_{p'_1,p'_2}$. Then it follows once again from Proposition~\ref{prop:ballsproperties}(ii) that 
\begin{equation*}
\max(d(p',p_1),d(p',p_2)) \leq \max_{i,j=1,2} d(p_i,p'_j)\leq \diam F
\end{equation*}
which concludes the proof.
\end{proof}
\smallskip

Given $F\subset\hn$ measurable, its Steiner symmetrization $\st F$
with respect to the $\Cn$-plane is defined by
\begin{equation*}
\st F =\{[z,t]\in\hn~;~z\in \pi(F),~2|t|\leq \lone (\{s\in\R~;~[z,s]\in F\})\}
\end{equation*} 
where $\lone$ denotes the one-dimensional Lebesgue measure and
$\pi:\hn\rightarrow\Cn$ is the canonical projection defined in \eqref{e:projection}. We define the reflection map
$\si:\hn\rightarrow\hn$ as 
\begin{equation}
  \label{eq:sigma}
\si([z,t])=[\overline z, t].  
\end{equation}

For the sake of simplicity, the following lemma is stated for compact sets. This will be the only case needed in this paper. It can however be easily generalized to non compact sets.

\begin{lem} \label{lem:diamstE}
 Let $F\subset \hn$ be compact and such that $\si(F) = F$. Then $\diam
 (\st F)\leq \diam F$. 
\end{lem}

\begin{proof}
Since $F$ is a compact subset of $\hn$, then $\tco F$ is compact and
is obviously $t$-convex. We have $\si(\tco F) = \tco F$ as
soon as $\si(F) = F$. Since $\st F \subset \st (\tco F)$ and $\diam
(\tco F) = \diam F$ by Lemma~\ref{lem:tco}, it is thus sufficient to
consider $t$-convex compact sets $F$ such that $\si(F) = F$. Then we can describe $F$ as  
\begin{equation*}
F =\{[z,t]\in\hn~;~z\in \pi(F),~ a(z) - \mathcal{L}(z) \leq t \leq a(z) + \mathcal{L}(z)\}
\end{equation*} 
for some map $a:\pi(F) \rightarrow \R$ which satisfies $a(\overline z) = a(z)$ for all $z\in \pi(F)$ and where $\mathcal{L}(z) = \lone (\{s\in\R~;~[z,s]\in F\})/2$. Note that we also have $\mathcal{L}(\overline z) = \mathcal{L}(z)$ for any $z\in \pi(F)$.
\smallskip

Let $p_1 = [z_1,t_1]\in \st F$ and $p_2 = [z_2,t_2] \in \st F$. Set 
\begin{gather*}
 F_1 = \{[z,t]\in F~;~z=z_1~\text{or}~z=\overline z_1\}~, \\
F_2 = \{[z,t]\in F~;~z=z_2~\text{or}~z=\overline z_2\}~.
\end{gather*}
We will prove that
\begin{equation}\label{eq:l2.9-1}
 d(p_1,p_2) \leq \max(d(q_1,q_2)~;~q_1\in F_1,~q_2\in F_2)~.
\end{equation}
Since $F_1 \cup F_2 \subset F$ this will imply $d(p_1,p_2) \leq \diam F$ as wanted. 
\smallskip

Set $\hat F_i = [0,-a(z_1)]\cdot F_i$, $i=1,2$, i.e., 
\begin{gather*}
 \hat F_1 = \{z_1,\overline z_1\} \times [-\mathcal{L}(z_1),\mathcal{L}(z_1)]\\
\hat F_2 = \{z_2,\overline z_2\} \times [b-\mathcal{L}(z_2),b+\mathcal{L}(z_2)]
\end{gather*}
where $b= a(z_2) - a(z_1)$. The distance being left invariant, we have 
\begin{equation} \label{e:l2.9-1-1}
 \max(d(q_1,q_2)~;~q_1\in F_1,~q_2\in F_2) = \max(d(q_1,q_2)~;~q_1\in
 \hat F_1,~q_2\in \hat F_2)~. 
\end{equation}
Next set $T_2 = |b| + \mathcal{L}(z_2)$. We have $p_2\in \{z_2\} \times [-T_2,T_2]$ hence it follows from
Proposition~\ref{prop:ballsproperties}(ii) that  
\begin{equation}\label{eq:l2.9-2}
 d(p_1,p_2) \leq \max(d(p_1,[z_2, T_2]), d(p_1,[z_2, -T_2]))~.
\end{equation}
Assume that $b\geq 0$. Then $[z_2,T_2]\in \hat F_2$ and since $p_1 \in \hat F_1$ we get  
\begin{equation}\label{eq:l2.9-3}
 d(p_1,[z_2,T_2]) \leq \max(d(q_1,q_2)~;~q_1\in \hat F_1,~q_2\in \hat F_2)~.
\end{equation}
Let $\iota$ denote the isometry in $(\hn,d)$ defined by 
\begin{equation}
  \label{eq:iota}
\iota([z,t]) = [\overline z, -t].   
\end{equation}
We have $\iota (\hat F_1) = \hat F_1$ and $\iota([z_2,-T_2]) = [\overline z_2,T_2] \in \hat F_2$, hence
\begin{equation}\label{eq:l2.9-4}
 d(p_1,[z_2,-T_2]) = d(\iota(p_1),\iota([z_2,-T_2])) \leq \max(d(q_1,q_2)~;~q_1\in \hat F_1,~q_2\in \hat F_2)~.
\end{equation}
Inequalities \eqref{eq:l2.9-2}, \eqref{eq:l2.9-3} and \eqref{eq:l2.9-4} together with \eqref{e:l2.9-1-1} give
\eqref{eq:l2.9-1}. The case where $b\leq 0$ can be traited in a
similar way and this concludes the proof. 
\end{proof}

\section{Isodiametric problem}
\label{sec:isodiampb}

We recall the definitions of the isodiametric constant 
\begin{equation*}
 C_I = \sup \{ \leb (F)/(\diam F)^{2n+2}~;~0<\diam F <+\infty~\}
\end{equation*}
and of the class of compact isodiametric sets
\begin{equation*}
 \I = \{ E\subset \hn~;~E\text{ compact},~\diam E>0,~\leb (E) = C_I~(\diam E)^{2n+2} \}~.
\end{equation*}

Recall that it is not restrictive to ask isodiametric sets to be compact
as the closure of any set which realizes the supremum in the
right-hand side of the definition of $C_I$ is a compact set that still
realizes the supremum. 
\medskip

We also introduce the class of so-called rotationally invariant sets. Given $\theta = (\theta_1,\dots,\theta_n)\in\R^n$, we define the
rotation $r_\theta:\hn \rightarrow \hn$ around the $\R$-axis by
\[
r_\theta([z,t]) = [(e^{i\theta_1}z_1,\dots,e^{i\theta_n}z_n),t]~.
\]

Any
such $r_\theta$ is an isometry in $(\hn,d)$. We denote by $\rot$ the
class of rotationally invariant sets, 
\begin{equation*}
 \rot = \{ F\subset\hn~;~ r_\theta(F)\subset F \text{ for all } \theta
 \in \R^n \}~. 
\end{equation*}
We set
\begin{equation*}
 C_{I,\rot} = \sup \{ \leb (F)/(\diam F)^{2n+2}~;~F\in \rot,~0<\diam F <+\infty~\}
\end{equation*}
and denote by $\I_\rot$ the family of compact rotationally invariant sets that are isodiametric within the class of rotationally invariant sets,
\begin{equation*}
 \I_\rot = \{ E\in \rot~;~E\text{ compact},~\diam E>0,~\leb (E) = C_{I,\rot}~(\diam E)^{2n+2} \}~.
\end{equation*}

In other words, $\I$, resp. $\I_\rot$, denotes the class of compact sets, resp. compact sets in $\rot$, that maximize the $\leb$-measure among all subsets of $\hn$, resp. among all sets in $\rot$, with the same diameter. 
\smallskip

We first prove the existence of sets in $\I$ and $\I_\rot$.

\begin{thm} \label{thm:existence}
Both families $\I$ and $\I_\rot$ are nonempty. 
\end{thm}

\begin{proof}
The proof that $\I$ is non empty relies on the compactness
 of equibounded sequences of non empty compact sets with respect to the
Hausdorff metric (see 2.10.21 in~\cite{federer}) together with the
upper-semicontinuity of the Lebesgue measure (see Theorem 3.2 in
\cite{beer-villar}). 
The fact that $\I_\rot\neq \emptyset$ as well can be 
proved in a similar way noting that $\rot$ is closed with respect
to the convergence of sets in the Hausdorff metric. 
\end{proof}

The rest of this section is devoted to the study of the regularity of sets in $\I$ and $\I_\rot$. The necessary condition \eqref{nc} introduced in Section~\ref{sec:intro} will be one of the key ingredients in this study and we prove in the next proposition that sets in $\I$ and $\I_\rot$ do satisfy this condition. 

\begin{prop} \label{prop:neccdtrefined} 
 Let $E\in \I \cup \I_\rot$. Then $E$ satisfies the necessary
 condition \eqref{nc}. 
\end{prop}

\begin{proof}
We argue by contradiction. Assume that $E\in\I$ and $p\in\p E$ is such
that $d(p,q)<\diam E$ for all $q\in\p E$. The distance from $p$ being
a continuous and open map and $E$ being
compact, we have $\max_{q\in E} d(p,q) = \max_{q\in\p E} d(p,q)$ (see Lemma~\ref{lem:distopen} and Remark~\ref{rmk:distopen}) and hence  
\begin{equation*}
 \max_{q\in E} d(p,q) < \diam E~.
\end{equation*}
Choosing $r = (\diam E  -  \max_{q\in E} d(p,q))/2 >0$, it follows that 
\begin{equation*}
 \diam (E\cup \overline B(p,r)) = \diam E~.
\end{equation*}
On the other hand, since $E$ is closed and $p\in \p E$, we have
$\interior ( B(p,r) \setminus E) \not=\emptyset$. This implies in
particular that  
\begin{equation*}
 \leb(E\cup \overline B(p,r))>\leb(E)
\end{equation*}
which contradicts the fact that $E\in\I$.
\smallskip

When $E\in\I_\rot$ we modify the argument as follows. We set
\begin{equation*}
 F=E\cup \bigcup_{\theta\in\Rn} \overline B(r_\theta(p),r) 
\end{equation*}
where $r=(\diam E  -  \max_{q\in E} d(p,q))/2 >0$ as before. We have
$F\in \rot$ and 
\begin{equation}\label{diamFE}
 \diam F = \diam E~.
\end{equation}
Indeed, to prove \eqref{diamFE} we fix $q_1,q_2\in
F$. If $q_1,q_2\in E$ then $d(q_1,q_2)\leq\diam E$. If $q_1$, $q_2\in
F\setminus E$ then there exist $\theta_1$, $\theta_2\in \Rn$ such that
$d(q_i,r_{\theta_i}(p))\leq r$, $i=1,2$. Recalling that any rotation
$r_\theta$ is an isometry in $(\hn,d)$ and that $E\in\rot$, it follows that 
\begin{equation*}
\begin{split}
d(q_1,q_2)&\leq d(r_{\theta_1}(p),r_{\theta_2}(p))+2r\\
&=d(p,r_{\theta_1-\theta_2}(p))+2r\\
&\leq \max_{q\in E} d(p,q) + 2r = \diam E~.
\end{split}
\end{equation*}
If $q_1\in E$ and $q_2\in F\setminus E$ with $d(q_2,r_{\theta_2}(p))\leq r$ for some $\theta_2\in\Rn$, we have 
\begin{equation*}
 d(q_1,q_2) \leq d(q_1,r_{\theta_2}(p)) + r \leq d(p,r_{-\theta_2}(q_1)) +r \leq \max_{q\in E} d(p,q) + r \leq \diam E~.
\end{equation*}
On the other hand, similarly as before, we have 
\begin{equation*}
 \leb(F)>\leb(E)
\end{equation*}
and this contradicts the fact that $E\in\I_\rot$.
\end{proof}
\smallskip

Let us introduce some notations. Given a
compact set $E$, we define $f^+,f^-,U$ and $\hat E$ as follows. 
We set
\begin{align*}
f^+(z)&=\max(t\in\R~;~[z,t]\in E)~,\\
f^-(z)&=\min(t\in\R~;~[z,t]\in E)
\end{align*}
for all $z\in\pi(E)$. Clearly, $[z,f^\pm(z)] \in \p E$ for all $z\in \pi(E)$. Recalling the definition of $\tco E$ given in Subsection~\ref{subsec:geometric-transf}, one has 
\begin{equation*}
 \tco E = \{[z,t]\in \hn~;~ z\in \pi(E),~ f^-(z)\leq t\leq f^+(z)\},
\end{equation*}
and $\tco E$ is itself compact. 

We set 
\begin{equation*}
 U = \{ z\in \pi(E)~;~ f^-(z) < f^+(z) \}
\end{equation*}
and
\begin{equation*}
 \hat E = \{[z,t]\in \hn~;~z\in\overline U,~f^-(z)\leq t\leq f^+(z)\}~.
\end{equation*}

\smallskip
Since $\pi(E)$ is closed, we have $\overline U\subset  
\pi(E)$. In particular $f^+$ and $f^-$ are well defined on $\overline
U$. Moreover $\hat E = \tco E \cap \pi^{-1}(\overline U)$ is compact
and contained in $\tco E$. 

\smallskip
We are now ready to state our key regularity result. It concerns $t$-convex and compact sets satisfying \eqref{nc}. 

\begin{thm}  \label{thm:regboundary}
For any $t$-convex and compact set $E$ satisfying \eqref{nc} the
following properties hold.
\begin{itemize}
\item[(i)] The set $U$ is open in $\Cn$ and the maps $f^-$ and $f^+$
  are locally Lipschitz on $U$ and continuous on $\pi(E)$.
\smallskip

 \item[(ii)] $\leb(E)=\leb(\hat E)$.
\smallskip

 \item[(iii)] We have 
\begin{align*}
\interior E = \interior \hat E &= \{[z,t]\in \hn~;~z\in U,~f^-(z)<t< f^+(z)\}~,\\
\p \hat E &= \{[z,f^{\pm}(z)]\in\hn~;~ z\in \overline U\}~,\\
\hat E &= \overline{\interior E}~.
\end{align*}
\end{itemize}
\end{thm}

Before starting with the proof of Theorem~\ref{thm:regboundary}, we introduce some notations and give a technical lemma.
\smallskip

Given $r>0$ and $p_1$, $p_2\in \hn$ with $\pi(p_1) = \pi(p_2)$,
$p_1=[z_{12},t_1]$ and $p_2=[z_{12},t_2]$ for some $z_{12}\in\Cn$ and
$t_1$, $t_2\in\R$, we set $\delta_{12} = (t_2-t_1)/2$ and 
\begin{equation*}
 F_{p_1,p_2}(r) = \Big\{[w,s]\in\hn~;~\|w-z_{12}\|\leq
 r,~\Big|s-\dfrac{t_1+t_2}{2}\Big|\leq \delta_{12}\:
 \Big(1-\dfrac{\|w-z_{12}\|}{r}\Big)\Big\} 
\end{equation*}
whenever $\delta_{12}>0$. Clearly, $F_{p_1,p_2}(r_1)\subset
F_{p_1,p_2}(r_2)$ provided $r_1\leq r_2$.
\smallskip

\begin{lem} \label{lem:FisinB}
Let $C>0$, $d>0$ and $\delta>0$ be fixed. There exists
$\gamma(C,d,\delta)>0$ such that the following holds. For any
$r\in(0,\gamma(C,d,\delta)]$, $p_1=[z_{12},t_1]$ and
$p_2=[z_{12},t_2]\in \hn$ such that $\|z_{12}\|\leq C$ and
$\delta_{12}=\delta$, we have  
\begin{equation*}
 F_{p_1,p_2}(r) \subset \overline B(p,d)
\end{equation*}
for all $p\in\hn$ such that $p_1,p_2 \in \overline B(p,d)$.
\end{lem}

\begin{proof}
Set $p_0=[-z_{12},-(t_1+t_2)/2]$. After a left translation by $p_0$ we need to prove that for all $r>0$ small enough we have 
\begin{equation*}
 p_0\cdot F_{p_1,p_2}(r)\subset \overline B(p,d)
\end{equation*}
for all $p\in\hn$ such that $d(p,p_0 \cdot p_1)\leq d$ and $d(p,p_0
\cdot p_2)\leq d$. For such a $p = [z,t]\in\hn$ we have $(p_0 \cdot
p_i)^{-1}\cdot p=[z,t\pm\delta]\in\overline B(0,d)$ for $i=1,2$, hence
$\|z\|\leq d$ and $|t\pm\delta|\leq h_d(\|z\|)$. It follows that
$h_d(\|z\|)\geq |t|+\delta\geq \delta$. Recalling that $\lim_{u\to
  d^-}h_d(u)=0$ and considering $\overline r=\overline
r(d,\delta)\in(d/2,d)$ such that $h_d(u)< \delta$ whenever
$u\in(\overline r,d]$, we get that $\|z\|\leq \overline r$. 
\smallskip

We have
\begin{equation*}
 p_0\cdot F_{p_1,p_2}(r) = \Big\{[w,s]\in\hn~;~\|w\|\leq r,~|s+2 \im z_{12}\overline w|\leq \delta\Big(1-\dfrac{\|w\|}{r}\Big)\Big\}~.
\end{equation*}
Let $[w,s] \in p_0\cdot F_{p_1,p_2}(r)$ and let us show that $p^{-1}\cdot [w,s] \in \overline B(0,d)$ or equivalently that $\|w-z\|\leq d$ and 
\begin{equation*}
|s-t-2\im  z\overline w|\leq h_d(\|w-z\|)
\end{equation*}
provided $r>0$ is small enough. First note that $\|w-z\|\leq r+\overline r$ is less than $d$ provided $r\leq d-\overline r$. 
\smallskip

Next we have
\begin{equation*}
\begin{split}
s-t-2\im z\overline w \leq\: & \delta\Big(1-\dfrac{\|w\|}{r}\Big)-2 \im z_{12}\overline w -t -2\im z\overline w\\
\leq\: & (\delta-t)-\dfrac{\delta}{r} \:\|w\| +2\|z_{12}\|\:\|w\|+2\|z\|\:\|w\|\\
\leq\: & h_d(\|z\|)-\Big(\dfrac{\delta}{r} -(2C+2d)\Big)\:\|w\|
\end{split}
\end{equation*}
where the last inequality follows from the fact that $[-z,\delta- t]=[z,t-\delta]^{-1}\in\overline B(0,d)$. On the other hand if $r\leq (d-\overline r)/2$, we have $\max(\|w-z\|,\|z\|)\leq (d+\overline r)/2<d$. Let $M = M(d,\delta) >0$ denote the Lipschitz constant of $h_d$ on $[0,(d+\overline r)/2]$. Then we have 
\begin{equation*}
h_d(\|z\|) \leq h_d(\|w-z\|) + M\|w\|~.
\end{equation*}
It follows that 
\begin{equation*}
 s-t-2\im z\overline w \leq h_d(\|w- z\|) - \Big(\dfrac{\delta}{r} - (2C+2d+M)\Big)\:\|w\| \leq h_d(\|w- z\|)
\end{equation*}
provided $r\leq \min((d-\overline r)/2,\delta/(2C+2d+M))$.
\smallskip

Similarly we have
\begin{equation*}
\begin{split}
s-t-2\im z\overline w \geq\: & -\delta\Big(1-\dfrac{\|w\|}{r}\Big) -2 \im z_{12}\overline w-t -2\im z\overline w\\
\geq\: & -h_d(\|z\|)+\Big(\dfrac{\delta}{r}-(2C+2d)\Big)\|w\|\\
\geq\: & -h_d(\|w-z\|)+\Big(\dfrac{\delta}{r}-(2C+2d+M)\Big)\|w\|\\
\geq\: & -h_d(\|w-z\|)
\end{split}
\end{equation*}
provided $r\leq \min((d-\overline r)/2,\delta/(2C+2d+M))$ and where the second inequality follows from the fact that $[-z,-t-\delta]=[z,t+\delta]^{-1}\in\overline B(0,d)$. Hence the lemma follows with 
\begin{equation} \label{eq:gamma2}
\gamma(C,d,\delta)=\min\Big(\dfrac{d-\overline r(d,\delta)}{2}, \dfrac{\delta}{2C+2d+M(d,\delta)}\Big)~.
\end{equation}
\end{proof}

\begin{rmk}\label{rmk:monotonicitygamma3}
Note that, $C$ and $d$ being fixed, the function $\gamma(C,d,\delta)$
can be taken to be continuous w.r.t. the variable $\delta$. This is a
consequence of \eqref{eq:gamma2} and the fact that $\overline
r(d,\delta)$ can be chosen continuous w.r.t. $\delta$.
\end{rmk}

We turn now to the proof of Theorem \ref{thm:regboundary}.

\begin{proof}[Proof of Theorem \ref{thm:regboundary}]
  
We fix
a $t$-convex and compact set $E$ satisfying
\eqref{nc} and set  
\begin{equation*}
 C=\max_{p\in E} \|\pi(p)\| \quad\text{and}\quad d=\diam E~.
\end{equation*}
We begin with two lemmata. The first one is a consequence
of Lemma~\ref{lem:FisinB}.

\begin{lem}\label{lem:FinE} 
Let  $E$ be as above, $\delta>0$ be fixed and let $\gamma =
\gamma(C,d,\delta)>0$ be as in Lemma \ref{lem:FisinB}. Then, for any $r\in
(0,\gamma]$, one has 
\begin{equation*}
 F_{p_1,p_2}(r) \subset E
\end{equation*} 
for all $p_1$, $p_2\in E$ such that $\pi(p_1) = \pi(p_2)$ and $\delta_{12}=\delta$.
\end{lem}

\begin{proof}
Since $E$ is closed, it is enough to show that 
\begin{equation*}
 \interior F_{p_1,p_2}(r) \subset E.
\end{equation*}
Assume by contradiction that $(\interior F_{p_1,p_2}(r))
\setminus E$ is nonempty. By $t$-convexity of $E$ we know that also $E \cap \interior
F_{p_1,p_2}(r)$ is nonempty since it contains 
$L_{p_1,p_2}\setminus \{p_1,p_2\}$. Therefore there exists $p\in \p E \cap \interior
F_{p_1,p_2}(r)$. Then by \eqref{nc} there exists $q\in \p E$ such that
$d(p,q) = d$. Thanks to Lemma \ref{lem:FisinB}, we have 
\[
p\in \interior F_{p_1,p_2}(r) \subset B(q,d),
\]
thus $d(p,q) < d$ which is a contradiction.
\end{proof}

\begin{lem}\label{lem:onesideLip}
Let $E$ be as before and let $\delta>0$ be fixed. There exists
$\beta(C,d,\delta)>0$ such that the following holds. If $z\in U$ is such that $f^+(z)-f^-(z)= 2\delta$, then 
\begin{equation*}
f^+(w) \leq f^+(z) + \beta(C,d,\delta)\:\|w-z\|\quad\text{and}\quad
f^-(w) \geq f^-(z) - \beta(C,d,\delta)\:\|w-z\| 
\end{equation*}
for all $w\in\pi(E)$.
\end{lem}

\begin{proof}
We claim that the lemma holds for $\beta(C,d,\delta)=\alpha(d,\delta)+2C$ where $\alpha(d,\delta)$ is given by Lemma~\ref{coneball}. Assume by contradiction that $w\in\pi(E)$ is such that
\begin{equation*}
f^+(w) > f^+(z) + (\alpha(d,\delta)+2C)\:\|w-z\|~,
\end{equation*}
the case where $f^-(w) < f^-(z) - (\alpha(d,\delta)+2C)\:\|w-z\|$ being analogous. Since $[z,f^+(z)]\in\p E$, by \eqref{nc} there exists $q\in\p E$ such that $d([z,f^+(z)],q)=d$. Set $p=q^{-1}\cdot [z,f^+(z)]$ and $[w',s']=q^{-1}\cdot[w,f^+(w)]$. We have 
\begin{equation*}
p=[z-z_q,f^+(z)-t_q-2\im z_q\overline z]
\end{equation*}
and
\begin{equation*}
[w',s']=[w-z_q,f^+(w)-t_q-2\im z_q\overline{w}]~.
\end{equation*}
It follows that 
\begin{equation} \label{outofcone}
\begin{split}
 s' =\:& t_p+f^+(w)- f^+(z)-2\im z_q (\overline{w-z}) \\
 >\: & t_p+(\alpha(d,\delta)+2C)\|w-z\| -2 C \:\|w-z\|\\
 =\: & t_p+\alpha(d,\delta)\|w'-z_p\|
\end{split}
\end{equation}
where the inequality follows by the choice of $w\in\pi(E)$.
\smallskip

On the other hand set $p'=q^{-1}\cdot[z,f^-(z)]$. We have $z_{p'} =
z_p$. We also have $[z,f^-(z)]\in E$ and $q\in E$, hence
$d(q,[z,f^-(z)])\leq d$ or equivalently $p'\in \overline
B(0,d)$. Therefore
\begin{equation*}
 |t_{p'}|\leq h_d(\|z_p\|)~.
\end{equation*}
Recalling that $p\in \p B(0,d)$ we get
\begin{equation}\label{tpdelta}
t_p=h_d(\|z_p\|)
\geq \dfrac {t_p-t_{p'}}{2}  = \dfrac {f^+(z)-f^-(z)}{2} = \delta~.
\end{equation}
Then, taking Lemma~\ref{coneball} into account, we infer from
\eqref{outofcone} and \eqref{tpdelta} that $[w',s']\not \in \overline
B(0,d)$, i.e., $d(q,[w,f^+(w)])>d$. This implies that $\diam E> d$, a
contradiction. 
\end{proof}

\begin{rmk}\label{rmk:monotonicitybeta}
Taking into account Remark
\ref{rem:alpha}, one can
take the function $\beta(C,d,\delta)$ to be continuous w.r.t. the
variable $\delta$.
\end{rmk}

We prove now the continuity of $f^-$ and $f^+$ on $\pi(E)$.

\begin{lem} \label{lem:continuity}
The functions $f^-$ and $f^+$ are continuous on $\pi(E)$.
\end{lem}

\begin{proof}
Let $z\in\pi(E)$ and let us prove that $f^+$ is continuous at $z$, the case of the function $f^-$ being similar. Let $(z_j)$ be a sequence of points in $\pi(E)$ such that $z_j\to z$ as $j\to\infty$. Since $E$ is compact, $f^+$ is bounded and to prove the continuity of $f^+$ at $z$ it is sufficient to prove that any possible limit of the sequence $(f^+(z_j))$ coincides with $f^+(z)$. By contradiction assume that $f^+(z_j)\to t$ as $j\to\infty$ for some $t\not=f^+(z)$. Since $E$ is compact, hence closed, and $[z_j,f^+(z_j)]\in E$, we have $[z,t]\in E$. It follows in particular that, by definition of $f^+$ and $f^-$, we must have $f^+(z)> t\geq f^-(z)$. Setting $p_1 =
  [z,f^-(z)]$ and $p_2 = [z,f^+(z)]$, we thus have $\delta_{12} = (f^+(z)-f^-(z))/2>0$. Owing to Lemma~\ref{lem:FinE}, one obtains that $F_{p_1,p_2}(\gamma)\subset E$ where $\gamma = \gamma(C,d,\delta_{12})$ is given by Lemma~\ref{lem:FisinB}. Therefore, recalling once again the definition of $f^+$, we get in particular that $\liminf_{j\to\infty} f^+(z_j)\geq f^+(z)$, a contradiction.
\end{proof}

We turn now to the proof of the fact that $U$ is open and that $f^-$, $f^+$ are locally
 Lipschitz continuous on $U$. 

\begin{lem}\label{lem:UopenfLip}
 The set $U\subset\Cn$ is open and the maps $f^-$, $f^+$ are locally
 Lipschitz continuous on $U$. 
\end{lem}

\begin{proof}
Let us introduce some notations. We will denote by $\ball(z,r)$ the
open ball in $\Cn$ with center $z\in\Cn$ and radius $r>0$. Given
$z\in U$, we set 
\begin{equation*}
 \delta_z= \dfrac{f^+(z) - f^-(z)}{2}
\end{equation*} 
and 
\begin{equation*}
 \gamma_z=\gamma(C,d,\delta_z)~,\quad \beta_z=\beta(C,d,\delta_z)
\end{equation*} 
where $\gamma(C,d,\delta_z)$ and $\beta(C,d,\delta_z)$ are given, respectively, by Lemma~\ref{lem:FisinB} and Lemma~\ref{lem:onesideLip}.
\smallskip

We first prove that $U\subset\Cn$ is open. Let $z\in U$. Set $t_1=f^-(z)$, $t_2=f^+(z)$ and let $p_1=[z,t_1]$, $p_2=[z,t_2]\in E$. It follows from Lemma~\ref{lem:FinE} that $F_{p_1,p_2}(\gamma_z) \subset E$. In particular for any $w\in\Cn$ with $\|w-z\|< \gamma_z$, we have $[w,s]\in E$ for all 
\begin{equation*}
 \begin{split}
 s\in~  & \big[\dfrac{t_1+t_2}{2} - \delta_z\big(1-\dfrac{\|w-z\|}{\gamma_z}\big), \dfrac{t_1+t_2}{2} + \delta_z\big(1-\dfrac{\|w-z\|}{\gamma_z}\big)\big] \\
 &=  \big[t_1+\delta_z \dfrac{\|w-z\|}{\gamma_z}, t_2-\delta_z \dfrac{\|w-z\|}{\gamma_z}\big]~.
\end{split}
\end{equation*}
Since 
\begin{equation*}
t_1+\delta_z \dfrac{\|w-z\|}{\gamma_z} < t_2-\delta_z \dfrac{\|w-z\|}{\gamma_z}
\end{equation*}
as soon as $\|w-z\|< \gamma_z$, it follows that $\ball(z,\gamma_z)\subset U$. Hence $U\subset\Cn$ is open.
\smallskip

We prove now that $f^-$ and $f^+$ are locally Lipschitz continuous on
$U$. Let $z\in U$. By the previous argument we already know that 
\begin{equation} \label{eq:spoonriver}
f^+(w)\geq f^+(z)- \delta_z \dfrac{\|w-z\|}{\gamma_z}\quad\text{and}\quad f^-(w)\leq f^-(z)+\delta_z \dfrac{\|w-z\|}{\gamma_z}
\end{equation}
for all $w\in\ball(z,\gamma_z)$. Next, Lemma~\ref{lem:onesideLip}
implies that 
\begin{equation*}
f^+(w)\leq f^+(z)+\beta_z \|w-z\| \quad\text{and}\quad f^-(w)\geq f^-(z) - \beta_z \|w-z\|
\end{equation*}
for all $w\in\ball(z,\gamma_z)\subset \pi(E)$. 

\smallskip
We fix a compact set $K\subset U$ and define $L = \sup_{z\in K}
\max(\beta_z,\delta_z /\gamma_z)$ and $\gamma = \inf_{z\in K}
\gamma_z$. Owing to the continuity of $\delta_z$, 
$\beta_z$ and $\gamma_z$ (see Lemma \ref{lem:continuity} and
Remarks \ref{rmk:monotonicitygamma3} and \ref{rmk:monotonicitybeta}),
one has $L< +\infty$ and $\gamma>0$. Therefore, for any
$z,w\in K$ satisfying $\|z - w\|<\gamma$ we conclude that
\begin{gather*}
 |f^+(w)-f^+(z)|\leq L\,\|w-z\|~,\\
|f^-(w)-f^-(z)|\leq L\,\|w-z\|.
\end{gather*}
Hence $f^+$ and $f^-$ are Lipschitz continuous on $K$, as desired. 
\end{proof}

We are now ready to conclude the proof of Theorem~\ref{thm:regboundary}. Statement (i) follows from lemmata~\ref{lem:continuity} and~\ref{lem:UopenfLip}. Statement (ii) is a consequence of 
\begin{equation*}
 E \setminus \hat E = \{[z,t]\in \hn~;~ z\in \pi(E)\setminus\overline U,~t=f^+(z)\}
\end{equation*}
and of the continuity of $f^+$. Finally (iii) is straightforward
and follows from (i) by standard arguments.
\end{proof}
\smallskip

We are going to apply Theorem~\ref{thm:regboundary} to sets in $\I$ and $\I_\rot$. In order to do this, we first need to prove that such sets are $t$-convex.

\begin{prop} \label{prop:tconv-isodiam}
 Any set $E\in\I \cup \I_\rot$ is $t$-convex.
\end{prop}

\begin{proof}
 Assume by contradiction that one can find $p=[z,t]\in \tco E\setminus
E$. By definition of $\tco E$, one must have $f^-(z)<t<f^+(z)$ and hence $z\in U$.
\smallskip

We have $E\subset\tco E$ and $\diam \tco E = \diam E$ by
Lemma \ref{lem:tco}. Since $\tco E \in \rot$ whenever $E\in\rot$, this implies that $\tco E \in\I \cup \I_\rot$. Hence $\tco E$ is a compact set that satisfies \eqref{nc} (see Proposition~\ref{prop:neccdtrefined}) and is obviously $t$-convex. Then Theorem~\ref{thm:regboundary} applies to $\tco E$.  Noting that the maps $f^\pm$, and consequently the set $U$, associated to $E$ and $\tco E$ coincide, we get from Theorem~\ref{thm:regboundary}(iii) that $p\in\interior (\tco E)$.
\smallskip

Since $E$ is closed and $p\notin E$ it follows that $p \in
\interior (\tco E\setminus E)$. In particular $\interior (\tco
E\setminus E) \not= \emptyset$ and $\leb(\tco E\setminus
E)>0$. Recalling that $E\subset \tco E$ and that both $E$ and $\tco E$ belong to $\I$, resp. $\I_\rot$, with $\diam (\tco E) = \diam E$, this gives a contradiction.
\end{proof}

Noting that $\hat E \subset E$ whenever $E$ is $t$-convex and that $\hat E\in\rot$ whenever $E\in\rot$, we get from Theorem~\ref{thm:regboundary}(ii) that $\hat E \in \I$, resp. $\hat E\in \I_\rot$, whenever $E\in \I$, resp. $E\in \I_\rot$, with $\diam \hat E = \diam E$.
\smallskip

Summing up let us gather in the next theorem properties of sets in $\I \cup \I_\rot$ proved in this section. Recall that the notations used in the statement to follow are those introduced before Theorem~\ref{thm:regboundary}.

\begin{thm} \label{thm:regisodiam}
 For any set $E\in \I \cup \I_\rot$, the following properties hold.
\begin{itemize}
\item[(i)] $E$ is $t$-convex.
\smallskip

\item[(ii)] $\leb(\hat E) = \leb(E)$ and $\diam \hat E = \diam E$.
\smallskip

\item[(iii)] $\hat E \in \I$, resp. $\hat E\in \I_\rot$, whenever $E\in \I$, resp. $E\in \I_\rot$.
\smallskip

\item[(iv)] The set $U$ is open in $\Cn$ and the maps $f^-$ and $f^+$
  are locally Lipschitz on $U$ and continuous on $\pi(E)$.
\smallskip

 \item[(v)] We have 
\begin{align*}
\interior E = \interior \hat E &= \{[z,t]\in \hn~;~z\in U,~f^-(z)<t< f^+(z)\}~,\\
\p \hat E &= \{[z,f^{\pm}(z)]\in\hn~;~ z\in \overline U\}~,\\
\hat E &= \overline{\interior E}~.
\end{align*}
\end{itemize}

\end{thm}

\section{Isodiametric problem for rotationally invariant sets}
\label{sec:rotisodiamsets}

In this section we characterize the Steiner symmetrization with respect to the $\Cn$-plane of sets $E\in \I_\rot$. Our main result states that the set $\st E$ belongs to $\I_\rot$ and is uniquely determined once the diameter of $E$ is fixed. It coincides with a peculiar set $A_{\diam E}$, defined below, that consequently also belongs to $\I_\rot$. 
\smallskip

Constructing suitable pertubations of this set that preserve the diameter and the $\leb$-measure, see Proposition~\ref{prop:Aphi}, we also get two remarkable consequences. First, the essential non uniqueness of sets in $ \I_\rot$, see Corollary~\ref{cor:nonuniqueness}. Second, the existence of sets in $\I$ which are not rotationally invariant even up to isometries, see Corollary~\ref{cor:nonrotinvariant}.
\smallskip

Given $\lambda>0$, we set
\begin{equation*}
 l_\lambda(r) = \begin{cases}
         \tfrac{\lambda^2}{2\pi} \phantom{\,\frac{\lambda}{2})} \quad \text{if } r\in [0,\lambda/\pi]\\
         h_{\frac{\lambda}{2}}(r) \quad \text{if } r\in [\lambda / \pi,\lambda /2]\\
        \end{cases}
\end{equation*}
(see~\eqref{eq:accaerre} for the definition of $h_{\frac{\lambda}{2}}$) and 
\begin{equation*}
 A_\lambda = \{[z,t]\in\hn~;~2\|z\|\leq \lambda,~|t|\leq l_\lambda(\|z\|)\}~.
\end{equation*} 

We have 
\begin{equation*}
 A_\lambda = A_\lambda^1 \cup A_\lambda^2
\end{equation*}
where
\begin{equation*}
 A_\lambda^1 = \{[z,t]\in \hn~;~\|z\|\leq \lambda/\pi,~2\pi|t|\leq \lambda^2\}
\end{equation*}
and
\begin{equation*}
\begin{split}
 A_\lambda^2 &= \{[z,t]\in \hn~;~\lambda/\pi\leq \|z\|\leq \lambda/2,~|t|\leq h_{\frac{\lambda}{2}}(r)\} \\
&= \overline B(0,\lambda/2) \cap \{[z,t]\in \hn~;~\|z\|\geq \lambda/\pi\}~. 
\end{split}
\end{equation*}

The set $A_\lambda$ is the closed convex hull, in the Euclidean sense when identifying $\hn$ with $\R^{2n+1}$, of the ball $B(0,\lambda/2)$ in $(\hn,d)$ centered at the origin and with diameter $\lambda$. See Figure~\ref{ccball} in Section~\ref{sec:intro} for a picture. 
\smallskip

We first show that the diameter of $A_\lambda$ equals $\lambda$.

\begin{prop} \label{prop:diamA}
 We have $\diam A_\lambda = \lambda$ for all $\lambda>0$.
\end{prop}

\begin{proof} Since $A_\lambda = \delta_\lambda (A_1)$ for all $\lambda>0$, it is enough to prove that $\diam A_1 = 1$. We begin with a technical lemma.

\begin{lem} \label{lem:diamA}
 Let $d>0$, $r\leq d/\pi$ and $t\in\R$ be fixed. Set
\begin{equation*}
 C = \{[z,t]\in \hn~;~\|z\|=r\}
\end{equation*}
and
\begin{equation*}
D = \{[z,t]\in \hn~;~\|z\|\leq r\}~.
\end{equation*}
Let $p\in\hn$. Assume that $\diam (C \cup \{p\} ) \leq d$ and $\diam (D \cup \{p\} ) > d$. Then $\| z_p\|< r$ and there exists $q\in D\setminus C$ such that $z_p=z_q$ and $\diam (D \cup \{p\} )=d(p,q)$.
\end{lem}

\begin{proof} 
We have $d([z,t],[z',t])\leq d([z,t],[0,t]) + d([0,t],[z',t]) = \|z\| +\|z'\|  \leq 2r$ for all $[z,t]$, $[z',t] \in D$ and $d([z,t],[-z,t]) = 2\|z\| = 2r$ for all $[z,t]\in C$ (see~\eqref{e:d1}). It follows that $2r = \diam C = \diam D \leq \diam (C \cup \{p\} ) < \diam (D \cup \{p\} )$ hence $\diam (D \cup \{p\} ) = d(p,q)$ for some $q=[z_q,t]\in D\setminus C$.
\smallskip

Set $\overline d = \diam (D \cup \{p\} )$ and 
\begin{equation*}
 h^\pm_{p,\overline d}(z) = \pm \,  h_{\overline d}(\|z-z_p\|) + t_p + 2 \im z_p \overline z 
\end{equation*}
so that
\begin{equation*}
 \overline B(p,\overline d) = \{[z,s]\in \hn~;~\|z-z_p\|\leq \overline d,~ h^-_{p,\overline d}(z) \leq s \leq h^+_{p,\overline d}(z)\}~.
\end{equation*}
We have $q=[z_q,t]\in \p B(p,\overline d)$. Changing if necessary $p$, $C$ and $D$ into $\iota(p)$, $\iota(C)$ and $\iota(D)$ (see~\eqref{eq:iota} for the definition of $\iota$) one can assume with no loss of generality that $t = h^+_{p,\overline d}(z_q)$. 
\smallskip

We have $D \subset \overline B(p,\overline d)$. First we note that this implies $\|z_q-z_p\| <\overline d$. Otherwise, since $\|z_q\| <r$ we could find $z$ close enough to $z_q$ such that $[z,t]\in D$ with $\|z-z_p\| > \overline d$. On the other hand we have $\pi(D) \subset \pi(\overline B(p,\overline d)) = \{w\in\Cn~;~\|w-z_p\|\leq \overline d\}$ which gives a contradiction. It follows that the map $h^+_{p,\overline d}$ is well defined on an open neighbourhood of $z_q$ in $\Cn$. Next since $D \subset \overline B(p,\overline d)$ and $q$ is in the relative interior of $D$ in $\Cn \times \{t\}$ it follows that $h^+_{p,\overline d}$ admits a local minimum at $z_q$. Assume now by contradiction that $z_q\not = z_p$. Then $h^+_{p,\overline d}$ is differentiable at $z_q$ and we get that 
\begin{equation} \label{e:diff-h}
 \overline d \, h'(\dfrac{\|z_q-z_p\|}{\overline d}) \, \dfrac{z_q-z_p}{\|z_q-z_p\|} - 2 z_p^\perp = 0
\end{equation}
where $z_p^\perp = i z_p$. This implies that $z_q\not= 0$. Indeed otherwise \eqref{e:diff-h} implies that $z_p=0$ and hence $z_q=z_p$. Next if $z_p \not = 0$ we get from \eqref{e:diff-h} that $z_q = z_p + \scal{z_q}{z_p^\perp} \, z_p^\perp$ with $\scal{z_q}{z_p^\perp}\not=0$ where the scalar product is that of $\R^{2n}$ after identifying points in $\Cn$ with points in $\R^{2n}$. It follows that $\|z_q\| > \|z_p\|$ which also holds true if $z_p=0$. On the other hand, restricting to a segment $s\in (-\eps,\eps) \mapsto z_q + s(z_q-z_p)$ for $\eps>0$ small enough we get that $h''(\|z_q-z_p\|/\,\overline d)\geq 0$ hence $\|z_q-z_p\|/\,\overline d\leq r_c$ where $[0,r_c]$ is the interval where $h'$ is increasing (see Subsection~\ref{subsec:balls}). Since $\|z_q-z_p\|>0$ it follows that 
\begin{equation*}
 h'(\dfrac{\|z_q-z_p\|}{\overline d}) > h'(0) = \dfrac{2}{\pi}.
\end{equation*}
All together we finally get
\begin{equation*}
 \dfrac{2\overline d}{\pi} < \overline d \, h'(\dfrac{\|z_q-z_p\|}{\overline d})  = 2 \|z_p^\perp\| = 2 \|z_p\| < 2 \|z_q\| < 2r \leq \dfrac{2 d}{\pi}~.
\end{equation*}
Recalling that $d < \overline d$ this gives a contradiction and concludes the proof.
\end{proof}

We go back now to the proof of Proposition~\ref{prop:diamA}. We have $\diam A_1 = \diam \p A_1$ (recall Remark~\ref{rmk:distopen}). If we set
\begin{equation*}
 C^{\pm} = \{[z,t]\in \hn~;~\|z\|=1/\pi,~2\pi t=\pm 1\}~,
\end{equation*}
\begin{equation*}
 D^{\pm} = \{[z,t]\in \hn~;~\|z\|\leq 1/\pi,~2\pi t=\pm 1\}~,
\end{equation*}
 we have $\p A_1 \subset A_1^2 \cup D^+\cup D^- \subset A_1$ hence $\diam A_1 = \diam (A_1^2 \cup D^+\cup D^-)$. 
\smallskip

Let $p\in A_1^2$. We have $C^+ \subset A_1^2 \subset \overline B(0,1/2)$ hence $\diam (C^+ \cup \{p\}) \leq 1$. If $\diam (D^+ \cup \{p\}) >1$ it follows from Lemma~\ref{lem:diamA} that $\|z_p\|<1/\pi$ which contradicts the fact that $p\in A_1^2$. Hence $\diam (D^+ \cup \{p\}) \leq 1$ for all $p\in A_1^2$. Since $\diam A_1^2 = 1$ it follows that $\diam (D^+ \cup A_1^2) = 1$. Similarly we have $\diam (D^- \cup A_1^2) = 1$.
\smallskip

It thus only remains to check that $\diam (D^- \cup D^+) \leq 1$. Let $p\in D^-$. We have $C^+ \cup \{p\} \subset A_1^2 \cup  D^-$ hence $\diam (C^+ \cup \{p\}) \leq 1$ by the previous argument. If $\diam (D^+ \cup \{p\}) >1$ it follows from Lemma~\ref{lem:diamA} that one can find $q\in D^+$ such that $z_p=z_q$ and $d(p,q) = \diam (D^+ \cup \{p\} )>1$. On the other hand we have $- 2\pi t_p = 2\pi t_q= 1$ and it follows from~\eqref{e:d2} that $d(p,q) = 1$ which gives a contradiction. Hence $\diam (D^+ \cup \{p\}) \leq 1$ for any $p\in D^-$ and since $\diam D^-\leq 1$ we finally get $\diam (D^- \cup D^+) \leq 1$ as wanted. 
\end{proof}

The next lemma is an elementary remark that will be used later.

\begin{lem} \label{lem:isodiamsubsetcyl}
 Let $E\in \rot$. Assume moreover that $E$ is symmetric with respect to the $\Cn$-plane, i.e., $[z,-t]\in E$ for all $[z,t]\in E$. Then
\begin{equation*}
 E \subset \{[z,t]\in \hn~;~2\|z\|\leq \diam E,~2\pi |t|\leq (\diam E)^2\}~.
\end{equation*}
\end{lem}

\begin{proof}
 Let $[z,t]\in E$. Since $E\in\rot$ we have $[-z,t]\in E$ and it follows from~\eqref{e:d1} that $2\|z\| = d([-z,t],[z,t]) \leq \diam E$. We also have $[z,-t]\in E$ by assumption. It follows from~\eqref{e:d2} that $(2\pi |t|)^{1/2} = d([z,t],[z,-t])\leq \diam E$ and this concludes the proof.
\end{proof}

We give now the main result of this section. 

\begin{thm} \label{thm:rotisodiamsets}
 Let $E\in \I_\rot$. Then $\st E \in \I_\rot$ and $\st E = A_{\diam E}$.
\end{thm}

\begin{proof}
First we note that since $E$ is compact $\st E$ is also
compact. Indeed, since $E$ is bounded, $\st E$ is also obviously bounded. Next the fact that $E$ is compact implies that the map $z\in\pi(E) \mapsto \lone (\{s\in\R~;~[z,s]\in E\})$ is upper semi-continuous. It follows that $\st E$ is closed and hence compact. Since $E\in \rot$, we have $\si(E) = E$ and it follows from Lemma~\ref{lem:diamstE} that $\diam(\st E)\leq \diam E$. On the other hand one has $\leb(\st E) = \leb(E)$. Since $E\in \I_\rot$ and $\st E \in \rot$, one actually has $\diam(\st E) = \diam E$ and $\st E \in \I_\rot$.
\smallskip

We set $\lambda = \diam E$ and we prove now that $\st E = A_\lambda$. Noting that $\pi(\st E)=\pi(E)$ and taking into account the fact that $\st E \in \I_\rot$ is also symmetric with respect to the $\Cn$-plane, it follows from Theorem~\ref{thm:regisodiam} that 
\begin{equation*}
 \st E = \{[z,t]\in \hn~;~ z\in\pi(E),~ |t|\leq f(z)\}
\end{equation*}
for some continuous map $f: \pi(E) \rightarrow [0,+\infty)$ and that the set
\begin{equation*}
 U =\{z\in \pi(E)~;~ f(z)>0\}
\end{equation*}
is open in $\Cn$. 
\smallskip

We have $\pi(E) \subset \{z\in\Cn~;~2 \|z\|\leq \lambda\}$ by Lemma~\ref{lem:isodiamsubsetcyl} and we prove now that $f(z) \leq l_\lambda(\|z\|)$ for all $z\in \pi(E)$. In case $\pi(E) \cap \{z\in\Cn~;~2\|z\| = \lambda\} \not= \emptyset$, we note that since $U$ is open we must have $f(z) = 0 = l_\lambda(\|z\|)$ for any $z\in \pi(E)$ such that $2\|z\| = \lambda$. Next we know from Lemma~\ref{lem:isodiamsubsetcyl} that $f(z) \leq \lambda^2/(2\pi) = l_\lambda(\|z\|)$ for all $z\in\pi(E)$ such that $\|z\|\leq \lambda /\pi$. 
\smallskip

It thus only remains to prove that $f(z) \leq l_\lambda(\|z\|)$ for
$z\in\pi(E)$ such that $\lambda / \pi< \|z\|< \lambda / 2$. Given such a
$z$ we assume by contradiction that $f(z) > l_\lambda(\|z\|) \geq 0$ and set $p = [z,l_\lambda(\|z\|)] \in A_{\lambda} \cap \p B(0,\lambda/2)$. Since $f$ is continuous and $z$ belongs to the open set $U$, one can find an open set $V\subset U \subset \pi(E)$ containing $z$ and such that $f(\|w\|)>l_\lambda(\|z\|)$ for all $w\in V$. Since $\{[w,s]\in\hn~;~w\in V,~|s|\leq f(\|w\|)\}\subset\st E$ it follows that $p\in \interior (\st E)$. Next, since $p\in \p B(0,\lambda/2)$ with $\|z\| \in ( \lambda / \pi,\lambda / 2)$, we can write $p$ as 
\begin{equation*}
 p = [\:\chi \: e^{-i\vp} \:\dfrac{\sin\vp}{\vp}~,~\dfrac{2\vp - \sin(2\vp)}{2\vp^2}\,\|\chi\|^2\:]
\end{equation*}
for some $\chi\in\Cn$ such that $\|\chi\| = \lambda/2$ and some $\vp
\in (-\pi/2,\pi/2)$ (see e.g.~\cite{juillet}). It follows from~\cite[Lemma 1.11]{juillet} that $d(q,p)=2d(0,p)=\lambda$ where 
\begin{equation*}
 q = [\:-\chi\: e^{i\vp} \:\dfrac{\sin\vp}{\vp}~,~ - \:\dfrac{2\vp - \sin(2\vp)}{2\vp^2}\,\|\chi\|^2\:]~.
\end{equation*}
Since $p \in \interior (\st E)$ and since the distance function from $q$ is an open map (see Lemma~\ref{lem:distopen}) we can find $p'\in \st E$ such that $d(q,p')>\lambda$. On the other hand we have 
\begin{equation*}
 q = [\:e^{i(\pi+2\vp)} z, -\:l_\lambda(\|z\|)\:].
\end{equation*}
Since $\st E\in\rot$ is symmetric with respect to the $\Cn$-plane, we
get that $q\in \st E$, i.e. a contradiction.
\smallskip

It follows that $\st E \subset A_{\lambda}$. Since $\diam \st E = \diam E = 
\lambda = \diam A_{\lambda}$ by 
Proposition~\ref{prop:diamA}, and since $ A_{\lambda} \in \rot$ and
$\st E \in \I_\rot$, we get that $\leb(A_{\lambda} \setminus \st E)=
0$. Being $\st E$ closed, we obtain that $\interior(A_{\lambda})\setminus \st E = \interior(A_{\lambda} \setminus \st E) =\emptyset$. It then follows that $A_{\lambda} = \overline{\interior(A_{\lambda})} \subset \st E$ and finally $ \st E = A_{\lambda}$ as wanted. 
\end{proof}

We now show that $A_\lambda$ can be perturbed near the $t$-axis in such a
way that the resulting set has the same volume and diameter as
$A_\lambda$. As we shall see, the class of such perturbations
is quite rich. It contains in particular rotationally and non rotationally invariant sets. 

\begin{prop}\label{prop:Aphi}
There exists $r \in (0, 1/\pi)$ such that for every $\lambda>0$ and for any Lipschitz function $f:\C^n\to \R$ with compact support in $\{z\in \C^n;\ \|z\|<\lambda r\}$ and with Lipschitz constant $\lip(f) <\pi \lambda r/4$, the set 
\[
A_{\lambda,f} = \{[z,t]\in \hn;\; \|z\|\leq \lambda/\pi,\ 2\pi |t -
f(z)| \leq \lambda^2\} \cup A_\lambda^2 
\]
satisfies $\leb(A_{\lambda,f}) = \leb(A_\lambda)$ and $\diam
A_{\lambda,f} = \diam A_\lambda$.
\end{prop}

\begin{proof}
The fact that $\leb(A_{\lambda,f}) = \leb(A_\lambda)$ is a consequence of the definition
of $A_{\lambda,f}$ and of Fubini's theorem. To prove the last part of the
lemma, we assume that $\lambda=1$ without loss of generality. Since $A_1^2 \subset A_{1,f}$ and $\diam
A_1^2 = 1$ we have 
$\diam A_{1,f} \geq 1 = \diam A_1$. 
\smallskip

To complete the proof, we will show that the 
inequality $\diam A_{1,f} \leq 1 = \diam A_1$ holds up to a
suitable choice of $r$. First, for
a given $p\in \hn$, we define 
\[
h_{p}^\pm(z)  = \pm h(\|z-z_p\|) + t_p + 2 \im z_p \bar z,
\]
so that $\overline B(p,1) = \{[z,s]\in \hn~;~\|z-z_p\|\leq 1,~ h^-_{p}(z) \leq s \leq h^+_{p}(z)\}$. We set $\kappa = h'(0)/2 = 1/\pi$. 
\medskip

\textit{Claim $1$.} 
There exists $\overline r \in (0,1/2)$ such that, for all $p\in \hn$ such that $\|z_p\| \leq \kappa/4$, one has
\begin{equation} \label{e:eq+}
 h_{p}^+(z) \geq h_{p}^+(z_p) + \kappa\:\|z - z_p\| \;\; \text{for all } z\in\Cn \text{ such that } \|z-z_p\|\leq 2 \overline r
\end{equation}
and, similarly, 
\begin{equation} \label{e:eq-}
 h_{p}^-(z) \leq h_{p}^-(z_p) - \kappa\:\|z - z_p\| \;\; \text{for all } z\in\Cn \text{ such that } \|z-z_p\|\leq 2 \overline r.
\end{equation}

Indeed one can find $\overline r \in (0,1/2)$ such that 
\begin{equation*}
 |h(r) -h(0) - h'(0)\:r| \leq \dfrac{\kappa}{2}\:r
\end{equation*}
for all $r\leq 2\overline r$. Then we get 
\begin{equation*}
\begin{split}
 h_{p}^+(z) &= h(\|z-z_p\|) + t_p - 2 \scal{iz_p}{z-z_p}\\
&\geq h_{p}^+(z_p) + 2\kappa \:\|z-z_p\| -\dfrac{\kappa}{2}\:\|z-z_p\| - 2 \|z_p\| \|z-z_p\|\\
&\geq h_{p}^+(z_p) + \kappa\:\|z - z_p\|
\end{split}
\end{equation*}
for all $p\in \hn$ such that $\|z_p\| \leq \kappa/4$ and all $z\in\Cn$ such that $\|z-z_p\|\leq 2 \overline r$. Here the scalar product is that of $\R^{2n}$ after identifying points in $\Cn$ with points in $\R^{2n}$.
This gives \eqref{e:eq+}. The proof of \eqref{e:eq-} is similar.
\smallskip

We set $p_0= \big[0,1/(2\pi)\big]$.\smallskip

\textit{Claim $2$.} For all $\overline r >0$, there exists $\hat r>0$ such that
\begin{equation} \label{e:equation}
\text{if } p\in
B(p_0,\hat r)\cup B(p_0^{-1},\hat r) \text{ and }  q\in \partial A_{1}\setminus B(p,1)  \text{ then }  \|z_q\| < \overline r.
 \end{equation}

To prove this claim, we set 
\[
K = \partial A_1 \cap \{[z,t]\in \hn;\ \|z\|\geq \overline r\}\subset \overline B(p_0,1)\cap \overline B(p_0^{-1},1) 
\]
(recall that $p_0$, $p_0^{-1}\in A_1$ and $\diam A_1 =1$).
Since
\[
\partial B(p_0,1) \cap \partial A_1 = \{p_0^{-1}\} \quad \text{and} \quad \partial B(p_0^{-1},1) \cap \partial A_1 =\{p_0\}~,
\]
we actually have
\[ K \subset B(p_0,1)\cap B(p_0^{-1},1)~.
\]
By compactness of $K$ and continuity of $p \mapsto \max_{q\in K} d(p,q)$, one can then find $\hat r>0$ such that  
$K \subset B(p,1)\cap B(p',1)$ for any $p\in B(p_0,\hat r)$ and 
$p'\in B(p_0^{-1},\hat r)$. This proves Claim $2$.
\medskip

Fix $\overline r \in (0,1/\pi)$ such that \eqref{e:eq+} and \eqref{e:eq-} hold for all $p\in \hn$ satisfying $\|z_p\| \leq \kappa/4$. Then choose $\hat r>0$ such that \eqref{e:equation} holds. Set $r=\min(\overline r, 2\hat r/\pi, \kappa/4)$. Let $f:\C^n\to \R$ be a Lipschitz function
with compact support in $\{z\in \C^n;\ \|z\|< r\}$ and with Lipschitz constant 
$\lip(f) <\pi r/4$.
\smallskip

We have $\|f\|_\infty \leq \lip(f) \: r <  \pi r^2/4$ hence 
\begin{equation*}
 \begin{split}
  \partial A_{1,f} \setminus \partial A_1 &\subset \big\{[z,t]\in \hn~;~ \|z\|< r~,~ 4\:|t- \dfrac{1}{2\pi}|< \pi r^2\big\}  \\ 
&\phantom{....}\cup \big\{[z,t]\in \hn~;~ \|z\|< r~,~ 4\:|t+ \dfrac{1}{2\pi}|< \pi r^2\big\}\\
& \subset B(p_0,\pi r/2) \cup
B(p_0^{-1},\pi r/2)\\
&\subset B(p_0,\hat r)\cup B(p_0^{-1},\hat r)~.
 \end{split}
\end{equation*}

Now we take $p \in \partial A_{1,f}\setminus \partial A_1$ and $q\in \partial
A_{1,f}$. Without loss of generality, we
also assume that $p\in B(p_0,\pi r/2)$, the other case being analogous. Then 
\[p = [z_p,\frac{1}{2\pi} + f(z_p)]~.\]

If
$\|z_q\|\geq \overline r \geq r$ then $q\in \partial A_1$ and by \eqref{e:equation} we get $d(p,q)<
1$. If $\|z_q\|<\overline r < 1/\pi$ then 
\[t_q = \pm\frac{1}{2\pi} + f(z_q)~.\]
If $t_q = \dfrac{1}{2\pi} + f(z_q)$, since $\|z_q\|<\overline r$ and $|f(z_q)|< \pi r^2/ 4 \leq \pi \overline r^2/ 4$, we have $d(p_0,q)< \pi \overline r/2$. Therefore 
\[
d(p,q) < \dfrac{\pi}{2} (r+\overline r)  < 1~.
\]
If $t_q = - \dfrac{1}{2\pi} + f(z_q)$ we note that $\|z_p\| \leq \kappa/4$ and $\|z_p - z_q\| \leq r+\overline r \leq 2\overline r$. Then by \eqref{e:eq-} we get
\begin{equation*}
\begin{split}
 h_{p}^-(z_q) &\leq  h_{p}^-(z_p) - \kappa\:\|z_p - z_q\|\\
&= - \frac{1}{\pi} +t_p- \kappa\:\|z_p - z_q\|\\
&\leq - \frac{1}{2\pi} + f(z_p) - \lip(f)\:\|z_p - z_q\|\\
&\leq - \frac{1}{2\pi} +  f(z_q) = t_q~.
 \end{split}
\end{equation*}
Similarly, using \eqref{e:eq+}, 
\[
h_p^+(z_q) \geq h_p^+(z_p) + \kappa\:\|z_p - z_q\| \geq \frac{3}{2\pi} +
f(z_q) > t_q~.
\]
Hence we have 
\begin{equation*}
h_{p}^-(z_q) \leq t_q \leq h_p^+(z_q).
\end{equation*}
that is $q\in \overline B(p,1)$. 
\smallskip

It follows that $d(p,q)\leq 1$ for all $p \in \partial A_{1,f}\setminus \partial A_1$ and $q\in \partial
A_{1,f}$. Recalling that $\diam (\p A_1) = 1$ and that $\diam A_{1,f} = \diam (\p A_{1,f})$ this concludes the proof.

\end{proof}

We get from this proposition the following two consequences.
\smallskip

\begin{cor} \label{cor:nonuniqueness}
 There exists $E\in \I_\rot$ such that $p\cdot E \not= A_{\diam E}$ for all $p\in \hn$.
\end{cor}

In other words, although there is uniqueness modulo Steiner symmetrization with respect to the $\Cn$-plane for sets in $\I_\rot$ with a given diameter, we have essential non uniqueness of sets in $\I_\rot$. 

\begin{proof}
Consider a set $A_{\lambda,f}$ given by Proposition~\ref{prop:Aphi} for some $\lambda>0$ and where $f\not\equiv 0$ is moreover chosen in such a way that $A_{\lambda,f}\in \rot$. By Theorem~\ref{thm:rotisodiamsets}, we have $A_\lambda\in \I_\rot$. On the other hand, by Proposition~\ref{prop:Aphi}, we have $\leb(A_{\lambda,f}) = \leb(A_\lambda)$ and $\diam
A_{\lambda,f} = \diam A_\lambda$. Therefore $A_{\lambda,f}\in \I_\rot$. 
\smallskip

Let us prove that $p\cdot A_{\lambda,f} \not= A_\lambda$ for all $p\in\hn$. First we note that if $F \in \rot$ and $F$ is bounded then $p\cdot F \in \rot$ if and only if $z_p=0$. Next it is straightforward from the analytic description of $A_\lambda$ and $A_{\lambda,f}$ that $p\cdot A_{\lambda,f} \not= A_\lambda$ for any $p\in \hn$ such that $z_p=0$. This concludes the proof.
\end{proof}

Another consequence of Theorem~\ref{thm:rotisodiamsets} and Proposition~\ref{prop:Aphi} is the existence of non-rotationally invariant isodiametric sets, even modulo left translations.

\begin{cor} \label{cor:nonrotinvariant}
 There exists $E\in \I$ such that $p\cdot E \not\in \rot$ for all $p\in \hn$.
\end{cor}

\begin{proof}
 If $\I\cap\rot=\emptyset$ then there is nothing to prove. If  $\I\cap\rot\not=\emptyset$ then $\I_\rot = \I \cap \rot \subset  \I$. In particular it then follows from Theorem~\ref{thm:rotisodiamsets} that $A_\lambda \in \I$ for any $\lambda>0$. Let $A_{\lambda,f}$ be given by Proposition~\ref{prop:Aphi} where $f$ is moreover chosen in such a way that $A_{\lambda,f}\not\in \rot$. By Proposition~\ref{prop:Aphi}, we have  $\leb(A_{\lambda,f}) = \leb(A_\lambda)$ and $\diam
A_{\lambda,f} = \diam A_\lambda$, and hence $A_{\lambda,f}\in \I$. 
Next we check that $p\cdot A_{\lambda,f}\not\in \rot$ for all
$p\in\hn$. Assume by contradiction that $p\cdot A_{\lambda,f}\in \rot$ for some $p\in\hn$. Then we get in particular that $\rho_\theta(\pi(p\cdot A_{\lambda,f})) = \pi(p\cdot A_{\lambda,f})$ for all $\theta\in\R^n$ where $\rho_\theta$ is the rotation in $\Cn$ defined by $\rho_\theta(z)=(e^{i\theta_1} z_1,\cdots,e^{i\theta_n} z_n)$. On the other hand we have $\pi(p\cdot A_{\lambda,f}) = z_p + \{z\in\Cn~;~ 2\|z\|\leq \lambda\}$. All together this implies that $z_p=0$. Then we get that $A_{\lambda,f} = p^{-1}\cdot (p\cdot A_{\lambda,f})$ is the vertical left translation by $p^{-1} = [0,-t_p]$ of $p\cdot A_{\lambda,f} \in \rot$ and hence belongs to $\rot$, a contradiction.
\end{proof}



\end{document}